\documentclass[11pt]{amsart}
\usepackage{amsfonts,amssymb}
\usepackage{amsthm}
\usepackage{amscd}
\usepackage{array}
\usepackage{color}
\usepackage{pstricks,pst-node}
\usepackage{pst-plot}
\catcode`\@=11
\def\marginnote#1{}

\newcount\hour
\newcount\minute
\newtoks\amorpm
\hour=\time\divide\hour by60 \minute=\time{\multiply\hour by60
\global\advance\minute by-\hour}
\edef\standardtime{{\ifnum\hour<12 \global\amorpm={am}%
        \else\global\amorpm={pm}\advance\hour by-12 \fi
        \ifnum\hour=0 \hour=12 \fi
        \number\hour:\ifnum\minute<10 0\fi\number\minute\the\amorpm}}
\edef\militarytime{\number\hour:\ifnum\minute<10 0\fi\number\minute}
\renewcommand{\theequation}{\thesection.\arabic{equation}}

\def\draftlabel#1{{\@bsphack\if@filesw {\let\thepage\relax
      \xdef\@gtempa{\write\@auxout{\string
          \newlabel{#1}{{\@currentlabel}{\thepage}}}}}\@gtempa \if@nobreak
    \ifvmode\nobreak\fi\fi\fi\@esphack} \gdef\@eqnlabel{#1}}
    \def\@eqnlabel{}
\def\@vacuum{}
\def\draftmarginnote#1{\marginpar{\raggedright\scriptsize\tt#1}}

\def\draft{
  \oddsidemargin -.5truein
  \def\@oddfoot{\footnotesize \sl preliminary draft \hfil
    \rm\thepage\hfil\sl\today\quad\militarytime}
  \let\@evenfoot\@oddfoot \overfullrule 3pt
    \let\label=\draftlabel
    \let\marginnote=\draftmarginnote
  \def\@eqnnum{(\theequation)\rlap{\kern\marginparsep\tt\@eqnlabel}%
    \global\let\@eqnlabel\@vacuum}

  }

\voffset= - 0.5in \hoffset= - 1.0in

\def\be{\begin{equation}}
\def\ee{\end{equation}}
\def\bea{\begin{eqnarray}}
\def\eea{\end{eqnarray}}
\def\<{\langle}
\def\>{\rangle}

\def\tr{{\mathrm{tr\,}}}

\def\bea{\begin{eqnarray}}
\def\eea{\end{eqnarray}}

\def\beq{\begin{equation}}
\def\eeq{\end{equation}}
\def\ba{\beq\begin{array}{c}}
\def\ea{\end{array}\eeq}

\theoremstyle{plain}
\newtheorem{theorem}{Theorem}[section]
\newtheorem{lemma}[theorem]{Lemma}

\theoremstyle{definition}
\newtheorem{definition}{Definition}[section]
\newtheorem{remark}[definition]{Remark}

\let\text=\mathrm

\newcommand{\scr}{\scriptstyle}

\def\beq{\begin{equation}}
\def\eeq{\end{equation}}
\def\bea{\begin{eqnarray}}
\def\eea{\end{eqnarray}}

\def\Tr{{\tr}}

\newcommand{\cpict}[3]{
\dimen1=#1\advance\dimen1 by-\hsize\divide\dimen1 by-2 \vtop to #2{
\noindent\hskip\dimen1{\special{em:graph #3.bmp}} \vfil}\hskip-2cm }

\newcommand{\tcr}{\textcolor{red}}
\newcommand{\tcb}{\textcolor{blue}}

\newcommand{\sheet}[2]{{\stackrel{{#1}}{{#2}}}}

\long\def\rem#1{}
\let\@@savethanks\thanks
\def\thanks#1{\gdef\thefootnote{\alph{footnote}}\@@savethanks{#1}}

\begin{document}

\title[Algebras of quantum monodromy data and character varieties]
{Algebras of quantum monodromy data and decorated character varieties}
\author{Leonid Chekhov$^{\ast}$}\thanks{$^{\ast}$Steklov Mathematical Institute of Russian Academy of Sciences and
Laboratoire Poncelet, Moscow, Russia, and Michigan State University, East Lansing, USA.
Email: chekhov@mi.ras.ru.}
\author{Marta Mazzocco$^\dagger$}\thanks{$^\dagger$Department of Mathematical Sciences, Loughborough University, LE11 3TU, United Kingdom.  Email: m.mazzocco@lboro.ac.uk}
\author{Vladimir Rubtsov$^\star$}\thanks{$^\star$Universit\'e d'Angers, France and Theory Division, ITEP;
Moscow, Russia;  Email: volodya@univ-angers.fr}

\maketitle

\phantom{XXX}\hfill {\em For the 70th birthday of Nigel Hitchin}

\section{introduction}
The classical Riemann-Hilbert problem deals with Fuchsian systems on
the Riemann sphere.
Let us consider a meromorphic system of first order ODEs:
\begin{equation}\label{Fuchs}
\frac{d\Psi}{dz} = \sum_{i=1}^s \frac{A_i}{z-a_i}\Psi,
\end{equation}
 where $z$ is a coordinate on the sphere
$\Sigma_{0,s}:=\mathbb P^1\backslash  \{a_1,\ldots,a_s\}$, where $\{A_1,\ldots,A_s\}\subset {\mathfrak{sl}}_k(\mathbb C)$  are constant in $z$.

The Riemann-Hilbert correspondence is defined by associating to each Fuchsian system its monodromy representation class obtained by considering the analytic continuation of a fundamental matrix $\Psi(z)$ of  \eqref{Fuchs} around loops $\gamma_i$, $i=1,\dots,s$, encircling each singular point $a_i$:
$$
\rho: \pi_1(\mathbb P^1\setminus\{a_1,\ldots,a_s\}, a_0)\to SL_k(\mathbb C).
$$
Taking the conjugacy classes one obtains the Betti moduli space of monodromy representations, or the {$SL_k$--character variety}:
\begin{equation}\label{Betti-0}
\mathcal M_B={\rm Hom}\left(\pi_1 (\Sigma_{0,s}) \to SL_k (\mathbb C)\right)\slash_{SL_k (\mathbb C)}.
\end{equation}
Geometrically, the system \eqref{Fuchs} can be replaced by  the meromorphic connection
$$
\nabla := \left(\frac{\partial}{\partial z}-A(z)\right)dz = d -\sum_{i=1}^s \frac{A_i}{z-a_i }dz.
$$
on the trivial holomorphic vector bundle $\mathbb C^k\times\mathbb P^1\to \mathbb P^1\setminus\{a_1,\ldots,a_s\}$. In this setting
 the Riemann-Hilbert correspondence is an isomorphism
\begin{equation}\label{RH}
RH:\mathcal M_{DR}\simeq \mathcal M_B
\end{equation}
where $\mathcal M_{DR} $ is the   de Rham moduli space
$$
\mathcal M_{DR} = \left\{(\nabla, E), E\to \Sigma_{0,s}\right\}\slash_{\mathcal S}
$$
of  logarithmic connections  $\nabla$ on holomorphic rank $k$ vector bundles $E$ over the Riemann sphere $\Sigma_{0,s}$ with $s$ boundary components and $\mathcal S$ is the gauge group.

In \cite{Hit}, Hitchin proved that this map  is a symplectomorphism.
To be precise, denoting by $\mathcal O_i$ the conjugacy classes of the residues $A_i$, $i=1,\dots,s$, one can endow $\mathcal M_{DR}$ with the standard Lie--Poisson structure on ${\mathcal O}_1 \times \dots \times {\mathcal O}_s\in\mathfrak{sl}_k\times  \mathfrak{sl}_k\times\dots\times \mathfrak{sl}_k$ obtained by identifying $\mathfrak{sl}_k$ with $\mathfrak{sl}_k^*$. Upon fixing the conjugacy classes, this Poisson structure restricts to a symplectic structure.
On the Betti moduli space $\mathcal M_{B}$, Hitchin considered the Poisson structure constructed by Audin \cite{Audin} as follows. Consider the Atiyah--Bott symplectic structure
$$
\Omega = \frac{k}{4\pi}{\rm Tr}\int_{\Sigma}{\delta A}\wedge {\delta A}
$$
on the space ${\mathcal Conn}(\Sigma)$ of all smooth $\mathfrak g$-valued connections $A$ (for $\mathfrak g$ a simple Lie algebra) on a compact Riemann surface
 $\Sigma$. When no boundaries are present, one replaces the space ${\mathcal Conn}(\Sigma)$ by the space ${\mathcal M}(\Sigma)={\mathcal M_0}(\Sigma)/\mathcal S$, the quotient of the space
${\mathcal M_0}(\Sigma)$ of all flat connections on $\Sigma$ by the gauge  group. Since $\Sigma$ is closed, the momentum map is just the curvature, so that the space ${\mathcal M}(\Sigma)$ is just a reduced level set of the momentum mapping and thus a symplectic manifold. In the presence of boundaries, the curvature is the momentum map of a smaller group, so one needs to consider a central extension of the group of gauge transformations in order to construct the Poisson structure. In particular one needs to add a correction term to the Atiyah--Bott symplectic structure:
 $$
\Omega_c = \frac{k}{4\pi}{\rm Tr}\int_{\partial\Sigma}\phi \wedge {\delta A},\qquad {\rm d}_A\phi=  {\delta A},
$$
The Poisson structure on the Betti moduli space is the result of the Hamiltonian reduction on the zero level of the momentum map associated to $\Omega+\Omega_c$. This Poisson structure coincides with the Goldman bracket on the character variety $\mathcal M_B$ \cite{Gold,gold1}.

In this paper, we address the question of what happens to this theory if we allow connections with higher order poles on holomorphic rank $k$ vector bundles $E$ on Riemann surfaces $\Sigma_{g,s}$ of genus $g$ and $s$ boundary components.

This question has been addressed by a number of Hitchin's disciples. In \cite{PB1} Boalch treated the case of a system with a regular pole and a pole of order two at infinity by using the Laplace transform \cite{Har, Dub, M1} to map it back to the Fuchsian case. More recently he introduced the notion of wild character variety \cite{PB} in which higher order poles are blown up to produce extra regular poles (at the intersection between a boundary circle and the Stokes directions) and the fundamental group is replaced by the groupoid of closed loops around these extra regular poles. He defined the Poisson structure on the wild character variety by using the quasi-Hamiltonian approach by Alexeev and collaborators \cite{AKSM,AMM}.

An elegant approach was proposed by Gualtieri--Li-Pym \cite{GLP}. In this case they start with the space of meromorphic connections on a smooth curve with a pole divisor $D$ (poles with multiplicity). Let $S$ be the space such that the parallel transport defined by this connection with poles bounded by the divisor $D$ exists and is holomorphic (outside of the poles).  They then consider the Lie algebroid $\mathcal A$ tangent to $S$ and define the {\it Stokes groupoid}\/ as the Lie groupoid that integrates this Lie algebroid.  They compute this groupoid in the case of the Airy equation and demonstrate that it is the usual pairing groupoid with a twist. A similar computation can be carried out for every case in which total summability works. This result gives a beautiful geometric explanation of Ecalle resummability theory.

In \cite{CMR}, based on the idea of interpreting higher order poles in the connection as boundary components with {\it bordered cusps}\footnote{We use the term bordered cusp meaning a vertex of an ideal triangle in the Poincar\'e metric in order to distinguish it from standard cusps (without borders) associated to punctures on a Riemann surface.}  \cite{ChM1} on the Riemann surface, we introduced the notion of {\it decorated character variety.} Let us remind this definition here.

Topologically speaking a Riemann surface  $\Sigma_{g,s,n}$  of genus $g$ with $s$ holes and $n$ bordered cusps is equivalent to a Riemann surface $\tilde\Sigma_{g,s,n}$ of genus $g$, with $s$ holes and $n$ marked points $m_1,\dots,m_n$ on the boundaries. Then one defines {\it the fundamental groupoid of arcs}\/ $\pi_{\mathfrak a}(\Sigma_{g,s,n})$ as the set of all directed paths $\gamma_{ij}:[0,1]\to\tilde \Sigma_{g,s,n}$ such that $\gamma_{ij}(0)=m_i$ and $\gamma_{ij}(1)=m_j$ modulo homotopy. The groupoid structure is dictated by the usual path--composition rules.
 The $SL_k$ decorated character variety is defined as:
\begin{equation}\label{fga}
\mathcal M_{g,s,n}^k:={\rm Hom}\left(
\mathfrak \pi_{\mathfrak a}(\Sigma_{g,s,n}),SL_k(\mathbb C)
\right)\slash_{\prod_{j=1}^n U_j},
\end{equation}
where $U_j$ is the unipotent Borel subgroup in $SL_k(\mathbb C)$ (one unipotent Borel subgroup for each bordered cusp).

Our interest in the representation spaces and their interpretation as decorated character varieties goes back to study of the moduli space of monodromy representations for the fundamental group of the 4-holed sphere. This is the $SL_2-$character variety
\begin{equation}\label{Betti-0,4}
\mathcal M_B={\rm Hom}\left(\pi_1 (\Sigma_{0,4}) \to SL_2 (\mathbb C)\right)\slash_{SL_2 (\mathbb C)},
\end{equation}
and the above discussed Poisson structure can also be obtained as a reduction of so-called {\it Korotkin--Samtleben} bracket (\cite{KS}) on ${\rm Hom}\left(\pi_1 (\Sigma_{0,4}) \to SL_2 (\mathbb C)\right)$ (which is in fact the quasi-Poisson structure in sense of  \cite{AKSM,AMM}).  The relation with Teichm\"{u}ller space parametrisation and a quantisation of this Poisson character variety was proposed by the first two authors in \cite{ChM}.

This Poisson manifold is also known as the monodromy manifold of the linear system corresponding to the Painlev\'e $VI$ equation. Our notion of decorated character variety was motivated by a
challenging problem of giving a  definition compatible with the confluence operations that give rise to all other Painlev\'e differential equations - in this case  the {\it Stokes phenomenon} appears- the solutions in the vicinity of the multiple poles have different asymptotic behaviours in different sectors.

In this paper we study the Poisson structure on the representation space
$$
\mathcal R_{g,s,n}^k:={\rm Hom}\left(
\mathfrak \pi_{\mathfrak a}(\Sigma_{g,s,n}),SL_k(\mathbb C)
\right),
$$
induced by the Fock--Rosly bracket \cite{FR} as explained in \cite{AGS} (see also \cite{BR,BR1,RS}) and prove that the quotient by unipotent Borel subgroups giving rise to the decorated character variety \eqref{fga} is a Poisson reduction. More precisely, we consider the Poisson structure on the matrices
$M$ that correspond to directed arcs in $ \pi_{\mathfrak a}(\Sigma_{g,s,n})$.  We call these matrices {\it monodromy data,} because  they indeed contain Stokes matrices, connection matrices and standard monodromy matrices of linear systems on first order ODEs (see Section \ref{se:analytic}).
 We treat classical and quantum case simultaneously, thus providing a quantisation of the decorated character variety \eqref{fga}. It would be interesting to understand the categorical version of our quantisation along the lines of  the recent papers by Ben-Zvi, Brochier and Jordan \cite{BZBJ, BZBJ1} - we postpone this to future publications.

In their seminal papers \cite{FG1}, \cite{FG2}, Fock and Goncharov introduced a set of Darboux coordinates for $SL_k(\mathbb R)$ systems on Riemann
surfaces with holes. Nevertheless, to the best of our knowledge, a comprehensive analysis relating the Fock--Goncharov construction to the Fock--Rosly algebras was still missing. Elements of this construction (in the classical case, without references to Poisson or quantum structures) had appeared in papers of Musiker, Schiffler and Williams \cite{MSW2}, \cite{MW} mostly devoted to establishing connection to
cluster algebras; it was there where lambda-lengths were identified with upper-left elements of $SL_2(\mathbb R)$-monodromy data. Shear coordinates associated with decorated bordered cusps were introduced simultaneously and independently by the first two authors in \cite{ChM1} and by Allegretti \cite{Al}. A useful technical tool allowing avoiding most difficulties of the standard combinatorial description of structures on Riemann surfaces with holes is that, having at least one bordered cusp, we can consider ideal-triangle partitions of these surfaces with triangles based
only at bordered cusps enclosing all holes without cusps in monogons. We then restrict the set of mapping class group (MCG) transformations (and the corresponding cluster mutations) to {the generalized} cluster algebras introduced in \cite{ChSh}. Upon imposing these restrictions we can establish an isomorphism between extended shear coordinates and lambda-lengths, both enjoying homogeneous Poisson or quantum relations, and explicitly construct the monodromy data of the system (\ref{fga}) for $k=2$. The homogeneous Poisson or quantum relations for shear coordinates then induce classical or quantum Fock--Rosly relations for elements of monodromy data with 
the  Poisson  reduction imposed by the quotient by unipotent Borel subgroups. All matrix elements of all monodromy data are then sign-definite Laurent polynomials either of exponentiated shear coordinates or of lambda-lengths.

The paper is organized as follows. In Sec.~\ref{s:SL2}, we briefly recall the hyperbolic geometry description of Teichm\"uller spaces of
Riemann surfaces with holes and bordered cusps and formulate the main statement of the paper (Theorem~\ref{thm:1}). We then derive the
quantum Fock--Rosly-like algebras of monodromy data for $SL_2$ out of coordinate algebras of the quantum Teichm\"uller spaces.
In Sec.~\ref{s:SLk}, we study Poisson and quantum $R$-matrix structures of decorated character varieties for general $SL_k$-monodromy data  paying a special 
attention to the Poisson reduction due to factorization w.r.t. Borel subgroups. In Sec.~\ref{s:examples}, we consider in details three important
examples of our construction for a general $SL_k(\mathbb R)$-monodromy data: $\Sigma_{1,s+1,1}$, $\Sigma_{0,2,2}$, and $\Sigma_{0,1,3}$. Finally, in Sec.~\ref{s:RH},  keeping in mind the idea of extending the Riemann--Hilbert correspondence to  $\mathcal Z^{irr}_{g,s}\to\mathcal M_{g,s,n}^k$, where
 $\mathcal Z^{irr}_{g,s}$ is a suitably decorated moduli space of irregular connections $\nabla$ on a holomorphic rank $k$-vector
bundle $E \to \Sigma_{g,s}$, we propose a formula linking the number $n$ of bordered cusps on  $ \Sigma_{g,s,n}$ to the irregular type of the connection. Because, as explained in \cite{CMR}, the decorated character variety \eqref{fga} contains the wild character variety as the sub-algebra of functions that Poisson commute with the function associated to certain arcs connecting bordered cusps, the space  $\mathcal Z^{irr}_{g,s}$ is in fact an extension of the one considered by Boalch in  his  survey on Riemann--Hilbert correspondence in this same issue \cite{BoalHitProc}.

Throughout this paper we use the following notation: we denote by $M$ a monodromy datum, or a matrix in $SL_k$, and by $m_{ij}$ its elements. 

\section{Algebras for $SL_2(\mathbb C)$ monodromy data for surfaces with bordered cusps}
\label{s:SL2}

\subsection{Darboux coordinates in dimension $2$}

In the case of dimension $k=2$, the complex dimension of the decorated character variety is $6g -6+3 s+2 n$. In this section we restrict to the real-analytic sub-variety
$$
\mathcal R_{g,s,n}^{\mathbb R}:={\rm Hom}\left(
\mathfrak \pi_{\mathfrak a}(\Sigma_{g,s,n}),SL_2(\mathbb R)
\right),
$$
and show how to construct real Darboux coordinates. We then complexify them to define Darboux coordinates on the complex representation space $\mathcal R_{g,s,n}^2$.

Every bordered cusp is endowed with a \emph{decoration}---a horocycle based at the end of the cusp, which is a point on an absolute. Such a horocycle cuts out an infinite part of a cusp and we consider only parts of arcs that are confined between two horocycles decorating two cusps at which this arc terminates (this can be the same cusp, then the same horocycle).

We split the Riemann surface $\Sigma_{g,s,n}$ into ideal triangles based at bordered cusps; the edges of these triangles are  arcs; if there are holes without bordered cusps, these holes by prescription are always enclosed in monogons (obviously bordered by arcs starting and terminating at the same bordered cusp.

For every ideal-triangle decomposition as above we consider the dual fat graph $\Gamma_{g,s,n}$ all vertices of which are three-valent except exactly $n$ one-valent vertices being in 1-1 relation with bordered cusps. Each (non-directed) edge of $\Gamma_{g,s,n}$ carries an (extended) shear coordinate (denoted by capital $Z$ letters for internal edges and by $\pi_j$ for edges ending at one-valent vertices) from the set $\{Z_\alpha,\pi_j\}$; for every hole without cusps we have an edge dual to the bordering arc and the loop attached to the (inner) end of this edge carrying the coefficient $\omega$ ($\omega=e^{P/2}+e^{-P/2}$ for a hole with the perimeter $P$ or $\omega=2\cos(\pi/n)$ for a $\mathbb Z_n$ orbifold point). The shear coordinates $\{Z_\alpha,\pi_j\}$ are either real numbers (in the classical case) or Hermitian operators (in the quantum case) with constant commutation relations determined by the graph  $\Gamma_{g,s,n}$. All coefficients $\omega$ are Casimirs.

From \cite{ChSh}, for any choice of real numbers $\{Z_\alpha,\pi_j\}$ and parameters $\omega_\beta$, we have a metrisable Riemann surface $\Sigma_{g,s,n}$, and vice versa, for any Poincar\'e uniformizable $\Sigma_{g,s,n}$, we have a (non-unique) set of $\{Z_\alpha,\pi_j\}$ and $\omega_\beta$ and a graph $\Gamma_{g,s,n}$ determining the gluing of this surface out of ideal triangles using the extended shear coordinates
$\{Z_\alpha,\pi_j\}$.

Then, we have an explicit parameterisation of the monodromy data in terms of extended shear coordinates. To every edge we set into correspondence the \emph{edge matrix} $X_A:=
\left(
\begin{array}{cc}
  0 &  -e^{A/2}   \\
   e^{-A/2}  &  0
\end{array}
\right)
$
where $A\in \{Z_\alpha,\pi_j\}$ is the extended shear coordinate of this edge. When an oriented path goes through the corresponding edge (in any direction) we multiply from the left by the edge matrix. When a path
turns left or right at three-valent vertices, we multiply from the left by the corresponding matrices
$L=\left(
\begin{array}{cc}
  0 &  1   \\
   -1  &  -1
\end{array}
\right)
$
and
$R=\left(
\begin{array}{cc}
  1 &  1   \\
   -1  &  0
\end{array}
\right)
$
and when a path goes clockwise around the loop containing the hole endowed with the coefficient $\omega$ we multiply from the left by
$F_\omega=\left(
\begin{array}{cc}
  0 &  1   \\
   -1  &  -\omega
\end{array}
\right).
$
We always begin with the edge matrix $X_{\pi_i}$ of the cusp the arc begins with; the last matrix in the product is always the edge matrix $X_{\pi_j}$ of the cusp at which the arc terminates. In the quantum case, the quantum ordering is the natural ordering in the product provided we scale the left and right turn matrices: $L\to q^{1/4}L$, $R\to q^{-1/4}R$ (if $[X,Y]=2\pi i\hbar$, then $e^Xe^Y=q e^{X+Y}$, so $q=e^{\pi i\hbar}$).

The thus constructed matrix products are invariant under the quantum MCG transformations (mutations of inner edges including those dual to monogons containing holes without cusps) so we can always reduce a given matrix product to a simpler one: for example, a monodromy datum corresponding to an arc with ends at two different cusps and such that it does not correspond to a bordered arc can be brought to the form $M=X_{\pi_j}L X_Z R X_{\pi_i}$; every monodromy datum of a bordered arc that goes clockwise can be brought to the form $M=X_{\pi_j}L X_{\pi_i}$ and the one going counterclockwise can be brought to the form $M=X_{\pi_j} R X_{\pi_i}$.

We now define the monodromy data for arc-like paths.

\begin{definition}\label{def1}
For a given set of extended shear coordinates $\{Z_\alpha,\pi_j\}$ and coefficients $\omega_\beta$ associated with a spine (fat graph)
$\Gamma_{g,s,n}$, the $SL_2(\mathbb R)$-monodromy data associated to (directed) arcs (directed paths starting and
terminating at bordered cusps) are
\begin{equation}\label{eqM}
M_{\mathfrak a}=X_{\pi_2}L X_{Z_{\alpha_n}}R\cdots L X_{Z_{\alpha_j}} F_{\omega_\beta} X_{Z_{\alpha_j}}R\cdots L X_{Z_{\alpha_1}}R X_{\pi_1},
\end{equation}
where $\pi_1$ and $\pi_2$ are the extended shear coordinates of the respective starting and terminating bordered cusps. In the quantum case,
the quantum ordering is the natural ordering prescribed by matrix multiplication and we make a scaling $L\to q^{1/4} L$, $R\to q^{-1/4} R$.
\end{definition}

\begin{definition}\label{def2}
We define the $\lambda$-length of an arc ${\mathfrak a}$ to be the upper-left element of $M_{\mathfrak a}$ defined by (\ref{eqM})
(u.r.$(M)$ in \cite{MSW2}, \cite{MW} or tr$_K(M)$ in \cite{ChM1}).
\end{definition}

Identifying $\{Z_\alpha,\pi_j\}$ with the extended shear coordinates, the thus defined $\lambda$-lengths (in the classical case) are $e^{\ell_{\mathfrak a}/2}$, where $\ell_{\mathfrak a}$ are actual (signed)
lengths of stretched between decorating horocycles parts of geodesic curves that join the corresponding bordered cusps and belong to the same homotopy class as the arc $\mathfrak a$. We often just identify these arcs with the corresponding $\lambda$-lengths writing merely $\lambda_{\mathfrak a}$.

\begin{lemma}\label{lem:1}
\cite{ChPen,ChSh,ChM1}\ \ The classical and quantum monodromy data $M_{\mathfrak a}$ are invariant under MCG transformations (extended cluster mutations) induced by mutations (flips) of inner edges of $\Gamma_{g,s,n}$.
\end{lemma}

\begin{definition}\label{def3}
\cite{ChM1} \ \ A {\it complete geodesic lamination (CGL)}\/ is the set of all the edges of all ideal triangles constituting an ideal triangle decomposition of $\Sigma_{g,s,n}$ with vertices at bordered cusps. We call {\it algebraic} CGL the collection of all $\lambda$-lengths of the elements in the CGL.
\end{definition}

\begin{lemma}\label{lem:2}
\cite{ChM1}\ \ Every algebraic CGL can be identified with a seed of a quantum cluster algebra of geometric type \cite{BZ,ChM1}; the corresponding
$\lambda$-lengths enjoy homogeneous commutation relations among themselves.
\end{lemma}

\begin{lemma}\label{lem:3}
\cite{MSW2,MW,ChM1}\ \ For the fat graph $\Gamma_{g,s,n}$ dual to the corresponding partition of $\Sigma_{g,s,n}$ into ideal triangles,
the relations between $\lambda$-lengths from the corresponding CGL and exponentiated extended shear coordinates $\{e^{\pm Z_\alpha},e^{+\pi_j/2}\}$ are 1-1 and monoidal in both directions.
\end{lemma}

\begin{lemma}\label{lem:4}
\cite{ChM1}\ \ Every $\lambda$-length in every CGL (a seed) is a polynomial from $\mathbb Z_+[e^{\pm Z_\alpha},e^{+\pi_j/2},\omega_\beta]$ of exponentiated shear coordinates of any given seed. Every quantum $\lambda$-length is a Hermitian operator represented by an ordered polynomial from $\mathbb Z_+[e^{\pm Z_\alpha},e^{+\pi_j/2},\omega_\beta,q^{\pm 1/4}]$.
\end{lemma}

Note that  Lemmata~\ref{lem:3} and \ref{lem:4} then immediately implies the Laurent and positivity phenomenon for $\lambda$-lengths in all seeds. 

The main result of this paper is as follows.

\begin{theorem}\label{thm:1}
Provided the extended shear coordinates $\{Z_\alpha,\pi_j\}$ enjoy the standard constant commutation or Poisson relations \cite{ChF,ChM1}, the
set of (classical or quantum) monodromy data $M_{\mathfrak a}$ determined by formula (\ref{eqM}) for any (directed) arc $\mathfrak a$ from a fixed CGL satisfies the following properties:
\begin{enumerate}
  \item[{\bf(a)}]
  The (classical or quantum) matrices $M_{\mathfrak a}$ corresponding to arcs from the same CGL satisfy $R$-matrix permutation relations of
Fock--Rosly  type \eqref{basic1}, \eqref{basic1.1}, \eqref{rel-comp1}\, \eqref{rel-comp2}, \eqref{rel-comp3}, \eqref{rel-comp4}, \eqref{21-43}, \eqref{32-41}, \eqref{31-42}.
  \item[{\bf(b)}] All elements of every monodromy datum  $M_{\mathfrak a}$ are polynomials of $e^{\pm Z_\alpha/2}$, $e^{\pm\pi_j/2}$, $q^{\pm 1/4}$ and $\omega_\beta$ with integers coefficients and, by Lemma~\ref{lem:3}, are sign-definite Laurent polynomials in $\lambda$-lengths  of any given CGL (for any seed of the corresponding cluster algebra of geometric type) and polynomials in $\omega_\beta$.
  \item[{\bf(c)}] The Borel subgroup restriction (see (\ref{fga})) is realized by reducing all monodromy data 
  $M_{\mathfrak a}$ that  correspond to paths between two neighbouring bordered cusps (or the same cusp if a hole contains only one cusp) that
  go along the hole boundary with the hole being to the left to the form with entries $m_{{\mathfrak a}_{i,j}}=0,\forall i+j\ge k+2$. This restriction is Poisson  for any system of monodromy data (and survives the generalisation to the case of $SL_k(\mathbb R)$-monodromy data).
\end{enumerate}
\end{theorem}

Observe that thanks to Lemma \ref{lem:3} and point (b) of  Theorem \ref{thm:1},  the extended shear coordinates $\{Z_\alpha,\pi_j\}$ can be straightforwardly complexified and therefore  the complex representation space $\mathcal R_{g,s,n}^2$ is endowed with the same Poisson structure.

In the following subsections we derive point (a) of the theorem (points (b) and (c) will be clear from construction) in the quantum case (the classical one follows by taking the semi-classical limit).
Let us stress that while the final formulae for the quantum commutation relations have been derived from the Fock-Rosly bracket \cite{FR} already in \cite{AGS} (see also \cite{BR,BR1,RS}), our approach allows to express all matrices in Darboux coordinates. For this reason we repeat the derivation of the relations from \cite{AGS} here.

In the following, we call {\it open arc} an arc joining different cusps, we call {\it closed arc} an arc joining the same cusp.

\subsection{Basic relations}

In this subsection we use our geometric construction to find the two Fock--Rosly commutation relations from which all other commutation relations can be found using the groupoid property.

The first relation pertains to the case where two different arcs  $\mathfrak a_i$ and $\mathfrak a_j$ start at the same bordered cusp $\pi$ and then go to the left and to the right respectively never colliding again, see Fig.~\ref{fi:basic}.
\begin{figure}[h]
\begin{pspicture}(-2,-1.5)(2,2){
\newcommand{\PATTERNONE}{%
{\psset{unit=1}
\psline[linewidth=20pt,linecolor=black](0,0)(0,-1.2)
\psline[linewidth=12pt,linecolor=black](0,-1)(-0.8,-1.8)
\psline[linewidth=12pt,linecolor=black](0,-1)(0.8,-1.8)
\psline[linewidth=10pt,linecolor=black,linestyle=dashed](0.8,-1.8)(1.5,-1.8)
\psline[linewidth=10pt,linecolor=black,linestyle=dashed](0.8,-1.8)(0.8,-2.5)
\psline[linewidth=10pt,linecolor=black,linestyle=dashed](-0.8,-1.8)(-1.5,-1.8)
\psline[linewidth=10pt,linecolor=black,linestyle=dashed](-0.8,-1.8)(-0.8,-2.5)
\psline[linewidth=18pt,linecolor=white](0,0.1)(0,-1.2)
\psline[linewidth=10pt,linecolor=white](0,-1)(-0.8,-1.8)
\psline[linewidth=10pt,linecolor=white](0,-1)(0.8,-1.8)
\psline[linewidth=8pt,linecolor=white](-0.8,-1.8)(-1.6,-1.8)
\psline[linewidth=8pt,linecolor=white](-0.8,-1.8)(-0.8,-2.6)
\psline[linewidth=8pt,linecolor=white](0.8,-1.8)(1.6,-1.8)
\psline[linewidth=8pt,linecolor=white](0.8,-1.8)(0.8,-2.6)
\put(-0.2,0.1){$i$}
\put(0.07,0.1){$j$}
\rput(-0.5, -0.4){$_{\pi}$}
\rput(-.7, -1.1){$_{Z_1}$}
\rput(.7, -1.1){$_{Z_2}$}
\psline[linewidth=1.5pt,linecolor=red](-0.15,0)(-0.15,-1.05)
\psline[linewidth=1.5pt,linecolor=blue](0.15,0)(0.15,-1.05)
\psline[linewidth=1.5pt,linecolor=red](-0.93,-1.81)(-0.15,-1.03)
\psline[linewidth=1.5pt,linecolor=blue,linestyle=dashed](0.9,-1.8)(1.5,-1.8)
\psline[linewidth=1.5pt,linecolor=red,linestyle=dashed](-0.9,-1.8)(-1.5,-1.8)
\psline[linewidth=1.5pt,linecolor=blue](0.93,-1.81)(0.15,-1.03)

}
}
\rput(-5,1.5){\PATTERNONE}
}
\end{pspicture}
\caption{Fat graph on a Riemann surface with at least one cusp of coordinate $\pi$. We denote the standard shear coordinates by $Z_1$ and $Z_2$. The dashed part is the rest of the Riemann surface.}
\label{fi:basic}
\end{figure}
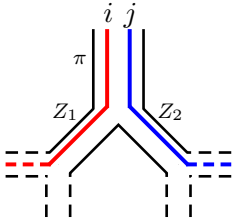

The matrices corresponding to arcs are:
\beq
M_i=Q X_{Z_1}RX_\pi;\qquad M_j=S X_{Z_2} L X_\pi,\qquad i<j
\label{MM}
\eeq
where we use the notation that $i<j$ when the arc $\mathfrak a_i$ is on the left of the arc $\mathfrak a_j$  and the matrices $Q$ and $S$ correspond to portions of the arcs that never intersect the arcs $\mathfrak a_j$ and $\mathfrak a_i$ respectively, so that
$$
[Q,M_j]=0, \quad [S,M_i]=0, \quad\hbox{and}\quad  [S,Q]=0.
$$
To deduce the commutation relations of the matrices (\ref{MM}) we use the following commutation relations:
$$
[Z_2,Z_1]=[Z_2,\pi]=[\pi,Z_1]=2\pi i\hbar,\
$$
$$
e^{Z_1/2}e^{\pi/2}=q^{-1/2} e^{\pi/2}e^{Z_1/2}, \quad e^{Z_2/2}e^{\pi/2}=q^{1/2} e^{\pi/2}e^{Z_2/2}, \quad e^{Z_2/2}e^{Z_1/2}=q^{-1/2} e^{Z_1/2}e^{Z_2/2}
$$
Then,  by a direct calculation, we obtain the following relations:
\bea
\label{basic1}\sheet{1}{M_i}\otimes \sheet{2}{M_j}=\sheet{2}{M_j}\otimes \sheet{1}{M_i}R_{12}(q)&\qquad i<j,\\
\label{basic1.1}\sheet{1}{M_i}\otimes \sheet{2}{M_j}R_{12}(q)^T=\sheet{2}{M_j}\sheet{1}{M_i}&\qquad i>j,
\eea
where
$R_{12}(q)=\sheet{1}e_{ii}\otimes \sheet{2}{e_{jj}}q^{(-1)^{i+j}}+\sheet{1}{e_{12}}\sheet{2}{e_{21}}(q^{1/2}-q^{-3/2})$ is the Kulish--Sklyanin $R$-matrix; explicitly
\beq
R_{12}(q)=
\left(
\begin{array}{cc|cc}
q^{1/2}  &  0 & 0  & 0\\
0  & q^{-1/2}  & q^{1/2}-q^{-3/2}  &0 \\
\hline
0  & 0  &  q^{-1/2} & 0\\
0  &0&0& q^{1/2}
\end{array}
\right)
\label{R12}
\eeq
Observe that as long as the arc $\mathfrak a_i$ is on the left of the arc $\mathfrak a_j$, by MCG transformations we can always flip edges in our fat-graph to match this situation.

We obtain the second basic relation for entries of the \emph{same} monodromy datum corresponding to an open arc: every such matrix (except the case where it borders a hole) can be brought by quantum MCG transformations to the form (see figure \ref{fi:basic2}):
$$
M:=X_{\pi_2} L X_{Z_1} R X_{\pi_1},\quad [\pi_1,Z_1]=[\pi_2,Z_1]=2\pi i \hbar,\ [\pi_1,\pi_2]=0,
$$
thus giving the following commutation relation:
\beq
R_{12}^{\text{T}} \sheet{1}{M}\otimes \sheet{2}{M}=\sheet{2}{M}\otimes \sheet{1}{M} R_{12}
\label{basic2}
\eeq
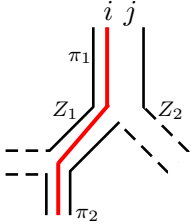
\begin{figure}[h]
\begin{pspicture}(-2,-1.5)(2,2){
\newcommand{\PATTERNONE}{%
{\psset{unit=1}
\psline[linewidth=20pt,linecolor=black](0,0)(0,-1.2)
\psline[linewidth=12pt,linecolor=black](0,-1)(-0.8,-1.8)
\psline[linewidth=12pt,linecolor=black,linestyle=dashed](0,-1)(0.8,-1.8)
\psline[linewidth=10pt,linecolor=black,linestyle=dashed](-0.8,-1.8)(-1.5,-1.8)
\psline[linewidth=10pt,linecolor=black](-0.8,-1.8)(-0.8,-2.5)
\psline[linewidth=18pt,linecolor=white](0,0.1)(0,-1.2)
\psline[linewidth=10pt,linecolor=white](0,-1)(-0.8,-1.8)
\psline[linewidth=10pt,linecolor=white](0,-1)(0.8,-1.8)
\psline[linewidth=8pt,linecolor=white](-0.8,-1.8)(-1.6,-1.8)
\psline[linewidth=8pt,linecolor=white](-0.8,-1.8)(-0.8,-2.6)
\psline[linewidth=8pt,linecolor=white](0.8,-1.8)(1.6,-1.8)
\psline[linewidth=8pt,linecolor=white](0.8,-1.8)(0.8,-2.6)
\put(-0.2,0.1){$i$}
\put(0.07,0.1){$j$}
\rput(-0.5, -0.4){$_{\pi_1}$}
\rput(-.7, -1.1){$_{Z_1}$}
\rput(.7, -1.1){$_{Z_2}$}
\rput(-0.4, -2.5){$_{\pi_2}$}
\psline[linewidth=1.5pt,linecolor=red](-0.15,0)(-0.15,-1.05)
\psline[linewidth=1.5pt,linecolor=red](-0.8,-1.81)(-0.13,-1.03)
\psline[linewidth=1.5pt,linecolor=red](-0.8,-1.8)(-0.8,-2.5)
}
}
\rput(-5,1.5){\PATTERNONE}
}
\end{pspicture}
\caption{Fat graph on a Riemann surface with two cusps of coordinates $\pi_1$ and $\pi_2$. We denote the standard shear coordinates by $Z_1,Z_2,\dots$. The dashed part is the rest of the Riemann surface.}
\label{fi:basic2}
\end{figure}

\subsection{Composite relations}
Here we explain how to obtain all other relations from the basic relations (\ref{basic1}) and (\ref{basic2}) using the groupoid property.

Using basic relation (\ref{basic1}) we can deduce what happens when two open arcs meet at two different cusps. Let these arcs be oriented in the same way, i.e. they originate at the same cusp and end at the same cusp. Inverting orientation corresponds to inverting a matrix. We denote the above two matrices by $M_i^j$ and $M_k^l$, where $i$ and $k$ are in the source cusp and $j,l$ in the target cusp; at each cusp
we have a linear ordering of indices originated from orientation of the surface. We can think of $M_i^j= M^{j^{-1}} M_i$ and $M_k^l= M^{l^{-1}} M_k  $ where:
$$
[M^j,M_k]=[M_i,M^l]=0,
$$  and the pairs $M_i,M_k$ and $M^j,M^l$ enjoy the commutation relations (\ref{basic1}). In this way we obtain
\bea
&&
\sheet{1}{M_i^j}\sheet{2}{M_k^l}= R_{12} \sheet{2}{M_k^l}\sheet{1}{M_i^j}  R_{12}\quad\hbox{for }i<k, j<l,\label{rel-comp1}\\
&&
\sheet{1}{M_i^j}\sheet{2}{M_k^l}= R_{12} \sheet{2}{M_k^l}\sheet{1}{M_i^j}  R_{12}^{-T}\quad\hbox{for }i>k, j<l,\label{rel-comp2}\\\
&&
\sheet{1}{M_i^j}\sheet{2}{M_k^l}= R_{12}^{-T} \sheet{2}{M_k^l}\sheet{1}{M_i^j}  R_{12}\quad\hbox{for }i<k, j>l,\label{rel-comp3}\\\
&&
\sheet{1}{M_i^j}\sheet{2}{M_k^l}= R_{12}^{-T} \sheet{2}{M_k^l}\sheet{1}{M_i^j}  R_{12}^{-T}\quad\hbox{for }i>k, j>l,\label{rel-comp4}\
\eea

Consider the case of two monodromy data corresponding to two closed arcs (starting and terminating at the same cusp) having no intersections inside the surface. Then their four ends can be uniquely enumerated from $1$ to $4$ depending on the order in which the corresponding arcs enter the cusp, see Fig.~\ref{fi:MMM}, where the index $1$ corresponds to the rightmost thread and $4$ to the leftmost thread. We have three different cases all of which can be obtained from basic relation (\ref{basic1}); $M_{ij}$ indicates the arc that starts at thread $i$ and terminates at $j$ having the structure $M^{-1}_j M_i$ with $M_i$ and $M_j$ from (\ref{MM}). In all examples below we take $i>j$, that is, the corresponding arc goes clockwise along the surface:
\begin{align}
&\sheet{1}{M_{21}} R_{12}\sheet{2} M_{43}R_{12}^{-1} =R_{12} \sheet{2}{M_{43}}R^{-1}_{12} \sheet{1}{M_{21}}
\label{21-43}
\\
&\sheet{1}{M_{41}} R_{12}^{-\text{T}}\sheet{2} M_{32}R_{12}^{\text{T}} =R_{12} \sheet{2}{M_{32}}R^{-1}_{12} \sheet{1}{M_{41}}
\label{32-41}
\\
&\sheet{1}{M_{31}} R_{12}^{-\text{T}}\sheet{2} M_{42}R_{12}^{-1} =R_{12} \sheet{2}{M_{42}}R^{-1}_{12} \sheet{1}{M_{31}}
\label{31-42}
\end{align}

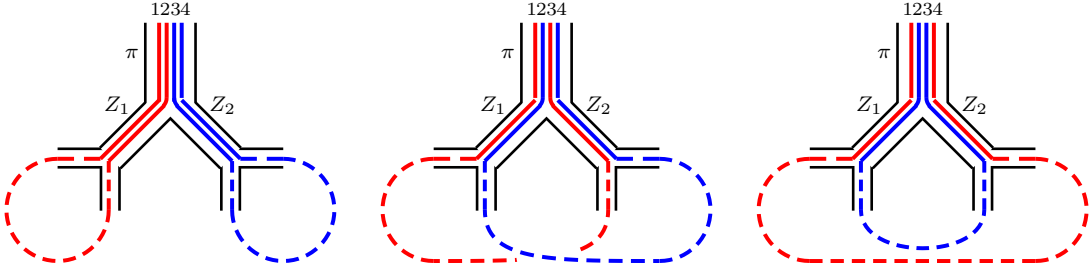
\begin{figure}[h]
\begin{pspicture}(-6,-2)(6,2){
\newcommand{\PATTERNONE}{%
{\psset{unit=1}
\psline[linewidth=20pt,linecolor=black](0,0)(0,-1.2)
\psline[linewidth=12pt,linecolor=black](0,-1)(-0.8,-1.8)
\psline[linewidth=12pt,linecolor=black](0,-1)(0.8,-1.8)
\psline[linewidth=8pt,linecolor=black](-0.8,-1.8)(-1.5,-1.8)
\psline[linewidth=8pt,linecolor=black](-0.8,-1.8)(-0.8,-2.5)
\psline[linewidth=8pt,linecolor=black](0.8,-1.8)(1.5,-1.8)
\psline[linewidth=8pt,linecolor=black](0.8,-1.8)(0.8,-2.5)
\psline[linewidth=18pt,linecolor=white](0,0.1)(0,-1.2)
\psline[linewidth=10pt,linecolor=white](0,-1)(-0.8,-1.8)
\psline[linewidth=10pt,linecolor=white](0,-1)(0.8,-1.8)
\psline[linewidth=6pt,linecolor=white](-0.8,-1.8)(-1.6,-1.8)
\psline[linewidth=6pt,linecolor=white](-0.8,-1.8)(-0.8,-2.6)
\psline[linewidth=6pt,linecolor=white](0.8,-1.8)(1.6,-1.8)
\psline[linewidth=6pt,linecolor=white](0.8,-1.8)(0.8,-2.6)
\put(-0.2,0.1){\makebox(0,0)[cb]{\hbox{{\tiny$1$}}}}
\put(-0.07,0.1){\makebox(0,0)[cb]{\hbox{{\tiny$2$}}}}
\put(0.07,0.1){\makebox(0,0)[cb]{\hbox{{\tiny$3$}}}}
\put(0.2,0.1){\makebox(0,0)[cb]{\hbox{{\tiny$4$}}}}
\rput(-0.5, -0.4){$_{\pi}$}
\rput(-.7, -1.1){$_{Z_1}$}
\rput(.7, -1.1){$_{Z_2}$}
\psline[linewidth=1.5pt,linecolor=red](-0.15,0)(-0.15,-1.05)
\psline[linewidth=1.5pt,linecolor=red](-0.05,0)(-0.05,-1)
\psline[linewidth=1.5pt,linecolor=blue](0.05,0)(0.05,-1.05)
\psline[linewidth=1.5pt,linecolor=blue](0.15,0)(0.15,-1)
\psarc[linewidth=1.5pt,linestyle=dashed,linecolor=red](-0.2,-1){0.15}{-45}{0}
\psline[linewidth=1.5pt,linecolor=red](-0.93,-1.81)(-0.15,-1.03)
\psline[linewidth=1.5pt,linecolor=red](-0.08,-1.1)(-0.82,-1.84)
\psline[linewidth=1.5pt,linestyle=dashed,linecolor=red](-0.93,-1.81)(-1.5,-1.81)
\psline[linewidth=1.5pt,linestyle=dashed,linecolor=red](-0.82,-2.49)(-0.82,-1.84)
\psarc[linewidth=1.5pt,linestyle=dashed,linecolor=blue](0.2,-1){0.15}{180}{225}
\psline[linewidth=1.5pt,linecolor=blue](0.93,-1.81)(0.15,-1.03)
\psline[linewidth=1.5pt,linecolor=blue](0.08,-1.1)(0.82,-1.84)
\psline[linewidth=1.5pt,linestyle=dashed,linecolor=blue](0.93,-1.81)(1.5,-1.81)
\psline[linewidth=1.5pt,linestyle=dashed,linecolor=blue](0.82,-2.49)(0.82,-1.84)
\psarc[linewidth=1.5pt,linestyle=dashed,linecolor=blue](1.5,-2.49){0.68}{-180}{90}
\psarc[linewidth=1.5pt,linestyle=dashed,linecolor=red](-1.5,-2.49){0.68}{90}{360}
}
}
\newcommand{\PATTERNTWO}{%
{\psset{unit=1}
\psline[linewidth=20pt,linecolor=black](0,0)(0,-1.2)
\psline[linewidth=12pt,linecolor=black](0,-1)(-0.8,-1.8)
\psline[linewidth=12pt,linecolor=black](0,-1)(0.8,-1.8)
\psline[linewidth=8pt,linecolor=black](-0.8,-1.8)(-1.5,-1.8)
\psline[linewidth=8pt,linecolor=black](-0.8,-1.8)(-0.8,-2.5)
\psline[linewidth=8pt,linecolor=black](0.8,-1.8)(1.5,-1.8)
\psline[linewidth=8pt,linecolor=black](0.8,-1.8)(0.8,-2.5)
\psline[linewidth=18pt,linecolor=white](0,0.1)(0,-1.2)
\psline[linewidth=10pt,linecolor=white](0,-1)(-0.8,-1.8)
\psline[linewidth=10pt,linecolor=white](0,-1)(0.8,-1.8)
\psline[linewidth=6pt,linecolor=white](-0.8,-1.8)(-1.6,-1.8)
\psline[linewidth=6pt,linecolor=white](-0.8,-1.8)(-0.8,-2.6)
\psline[linewidth=6pt,linecolor=white](0.8,-1.8)(1.6,-1.8)
\psline[linewidth=6pt,linecolor=white](0.8,-1.8)(0.8,-2.6)
\put(-0.2,0.1){\makebox(0,0)[cb]{\hbox{{\tiny$1$}}}}
\put(-0.07,0.1){\makebox(0,0)[cb]{\hbox{{\tiny$2$}}}}
\put(0.07,0.1){\makebox(0,0)[cb]{\hbox{{\tiny$3$}}}}
\put(0.2,0.1){\makebox(0,0)[cb]{\hbox{{\tiny$4$}}}}
\rput(-0.5, -0.4){$_{\pi}$}
\rput(-.7, -1.1){$_{Z_1}$}
\rput(.7, -1.1){$_{Z_2}$}
\psline[linewidth=1.5pt,linecolor=red](-0.15,0)(-0.15,-1)
\psline[linewidth=1.5pt,linecolor=blue](-0.05,0)(-0.05,-1)
\psline[linewidth=1.5pt,linecolor=red](0.05,0)(0.05,-1)
\psline[linewidth=1.5pt,linecolor=blue](0.15,0)(0.15,-1)
\psarc[linewidth=1.5pt,linecolor=blue](-0.2,-1){0.15}{-45}{0}
\psline[linewidth=1.5pt,linecolor=red](-0.93,-1.81)(-0.15,-1.03)
\psline[linewidth=1.5pt,linecolor=blue](-0.08,-1.1)(-0.82,-1.84)
\psline[linewidth=1.5pt,linestyle=dashed,linecolor=red](-0.93,-1.81)(-1.5,-1.81)
\psline[linewidth=1.5pt,linestyle=dashed,linecolor=blue](-0.82,-2.49)(-0.82,-1.84)
\psarc[linewidth=1.5pt,linestyle=dashed,linecolor=red](0.2,-1){0.15}{180}{225}
\psline[linewidth=1.5pt,linecolor=blue](0.93,-1.81)(0.15,-1.03)
\psline[linewidth=1.5pt,linecolor=red](0.08,-1.1)(0.82,-1.84)
\psline[linewidth=1.5pt,linestyle=dashed,linecolor=blue](0.93,-1.81)(1.5,-1.81)
\psline[linewidth=1.5pt,linestyle=dashed,linecolor=red](0.82,-2.49)(0.82,-1.84)
\psarc[linewidth=1.5pt,linestyle=dashed,linecolor=blue](1.5,-2.49){0.68}{-90}{90}
\psarc[linewidth=1.5pt,linestyle=dashed,linecolor=red](-1.5,-2.49){0.68}{90}{270}
\psbezier[linewidth=1.5pt,linestyle=dashed,linecolor=red](0.82,-2.49)(0.82,-3.17)(0,-3.17)(-1.5,-3.17)
\pscircle[linecolor=white,fillstyle=solid, fillcolor=white](-.,-3.1){.4}
\psbezier[linewidth=2.5pt,linestyle=solid,linecolor=white](-0.82,-2.49)(-0.82,-3.17)(0,-3.17)(1.5,-3.17)
\psbezier[linewidth=1.5pt,linestyle=dashed,linecolor=blue](-0.82,-2.49)(-0.82,-3.17)(0,-3.17)(1.5,-3.17)
}
}
\newcommand{\PATTERNTHREE}{%
{\psset{unit=1}
\psline[linewidth=20pt,linecolor=black](0,0)(0,-1.2)
\psline[linewidth=12pt,linecolor=black](0,-1)(-0.8,-1.8)
\psline[linewidth=12pt,linecolor=black](0,-1)(0.8,-1.8)
\psline[linewidth=8pt,linecolor=black](-0.8,-1.8)(-1.5,-1.8)
\psline[linewidth=8pt,linecolor=black](-0.8,-1.8)(-0.8,-2.5)
\psline[linewidth=8pt,linecolor=black](0.8,-1.8)(1.5,-1.8)
\psline[linewidth=8pt,linecolor=black](0.8,-1.8)(0.8,-2.5)
\psline[linewidth=18pt,linecolor=white](0,0.1)(0,-1.2)
\psline[linewidth=10pt,linecolor=white](0,-1)(-0.8,-1.8)
\psline[linewidth=10pt,linecolor=white](0,-1)(0.8,-1.8)
\psline[linewidth=6pt,linecolor=white](-0.8,-1.8)(-1.6,-1.8)
\psline[linewidth=6pt,linecolor=white](-0.8,-1.8)(-0.8,-2.6)
\psline[linewidth=6pt,linecolor=white](0.8,-1.8)(1.6,-1.8)
\psline[linewidth=6pt,linecolor=white](0.8,-1.8)(0.8,-2.6)
\put(-0.2,0.1){\makebox(0,0)[cb]{\hbox{{\tiny$1$}}}}
\put(-0.07,0.1){\makebox(0,0)[cb]{\hbox{{\tiny$2$}}}}
\put(0.07,0.1){\makebox(0,0)[cb]{\hbox{{\tiny$3$}}}}
\put(0.2,0.1){\makebox(0,0)[cb]{\hbox{{\tiny$4$}}}}
\rput(-0.5, -0.4){$_{\pi}$}
\rput(-.7, -1.1){$_{Z_1}$}
\rput(.7, -1.1){$_{Z_2}$}
\psline[linewidth=1.5pt,linecolor=red](-0.15,0)(-0.15,-1)
\psline[linewidth=1.5pt,linecolor=blue](-0.05,0)(-0.05,-1)
\psline[linewidth=1.5pt,linecolor=blue](0.05,0)(0.05,-1)
\psline[linewidth=1.5pt,linecolor=red](0.15,0)(0.15,-1)
\psarc[linewidth=1.5pt,linecolor=blue](-0.2,-1){0.15}{-45}{0}
\psline[linewidth=1.5pt,linecolor=red](-0.93,-1.81)(-0.15,-1.03)
\psline[linewidth=1.5pt,linecolor=blue](-0.08,-1.1)(-0.82,-1.84)
\psline[linewidth=1.5pt,linestyle=dashed,linecolor=red](-0.93,-1.81)(-1.5,-1.81)
\psline[linewidth=1.5pt,linestyle=dashed,linecolor=blue](-0.82,-2.49)(-0.82,-1.84)
\psarc[linewidth=1.5pt,linestyle=dashed,linecolor=blue](0.2,-1){0.15}{180}{225}
\psline[linewidth=1.5pt,linecolor=red](0.93,-1.81)(0.15,-1.03)
\psline[linewidth=1.5pt,linecolor=blue](0.08,-1.1)(0.82,-1.84)
\psline[linewidth=1.5pt,linestyle=dashed,linecolor=red](0.93,-1.81)(1.5,-1.81)
\psline[linewidth=1.5pt,linestyle=dashed,linecolor=blue](0.82,-2.49)(0.82,-1.84)
\psarc[linewidth=1.5pt,linestyle=dashed,linecolor=red](1.5,-2.49){0.68}{-90}{90}
\psarc[linewidth=1.5pt,linestyle=dashed,linecolor=red](-1.5,-2.49){0.68}{90}{270}
\psline[linewidth=1.5pt,linestyle=dashed,linecolor=red](1.5,-3.17)(-1.5,-3.17)
\psbezier[linewidth=1.5pt,linestyle=dashed,linecolor=blue](0.82,-2.49)(0.82,-3.17)(-0.82,-3.17)(-0.82,-2.49)
}
}
\rput(-5,1.5){\PATTERNONE}
\rput(0,1.5){\PATTERNTWO}
\rput(5,1.5){\PATTERNTHREE}
}
\end{pspicture}
\caption{\small Three cases of nonintersecting geodesic arcs terminating at the same bordered cusp. }
\label{fi:MMM}
\end{figure}

The last commutation relation is for the entries of the same matrix $M_{ij}$ corresponding to a closed arc. We can obtain this arc as the product of two arcs: one starts at the same cusp and terminates at another cusp and the second one starts at the second cusp and terminates at the first cusp going along a different path in the surface, in such a way that the composition gives the arc that starts and terminates at the same cusp. We obtain
\beq
R_{12}^{\text{T}}\sheet{1} {M_{ij}} R_{12}^{-\text{T}}\sheet{2}{M_{ij}}=\sheet{2}{M_{ij}} R_{12}^{-1} \sheet{1}{M_{ij}} R_{12}.
\label{H}
\eeq
Note that for the $R$-matrix of form (\ref{R12}), we have another, equivalent way of writing the {\em same} quantum commutation relations for elements of the matrix $M_{ij}$ encoded
in (\ref{H}):
\beq
\sheet{1} {M_{ij}} R_{12}^{-\text{T}}\sheet{2}{M_{ij}}R_{12}^{\text{T}}= R_{12}\sheet{2}{M_{ij}} R_{12}^{-1} \sheet{1}{M_{ij}}.
\label{H1}
\eeq
We are free to use any of relations (\ref{H}), (\ref{H1}) on our discretion.
\begin{remark}
That (\ref{H1}) is equivalent to (\ref{H}) implies the following commutation relation:
$$
R_{12}^{\text{T}}R_{12}\sheet{2}{M_{ij}} R_{12}^{-1} \sheet{1}{M_{ij}}=\sheet{2}{M_{ij}} R_{12}^{-1} \sheet{1}{M_{ij}}R_{12}R_{12}^{\text{T}}\ \hbox{and}\
R_{12}R_{12}^{\text{T}}\sheet{1}{M_{ij}} R_{12}^{-\text{T}} \sheet{2}{M_{ij}}=\sheet{1}{M_{ij}} R_{12}^{-\text{T}} \sheet{2}{M_{ij}}R_{12}^{\text{T}}R_{12}.
$$
\end{remark}

\section{Classical and quantum $R$-matrix structures of $SL_k(\mathbb C)$ monodromy data}
\label{s:SLk}

\subsection{$R$-matrix relations in the $SL_k(\mathbb C)$ case}

For generic dimension $k$, consider the following special case of trigonometric $R$-matrix generalizing the Kulish--Sklyanin matrix $R_{12}$ to the case of $SL_k(\mathbb C)$:
\beq
R_{12}(q)=\sum_{i,j}q^{-1/2}\sheet{1}{e_{ii}}\otimes \sheet{2}{e_{jj}}+\sum_i (q^{1/2}-q^{-1/2})\sheet{1}{e_{ii}}\otimes \sheet{2}{e_{ii}}+\sum_{j>i}
(q^{1/2}-q^{-3/2})\sheet{1}{e_{ij}}\otimes \sheet{2}{e_{ji}}
\label{Rn}
\eeq
Following  \cite{AGS}, we use the $R$-matrix relations (\ref{basic1}) and (\ref{basic2}) as well as all ``composite'' relations (\ref{21-43})--(\ref{H1}) to define a quasi-Poisson structure on the representation space $\mathcal R_{g,s,n}^k$. The following theorem shows that actually this is a Poisson structure - we call it {\it  Poisson algebra of monodromy data}:

\begin{theorem}\label{Jacobi}
For any $R$-matrix $R_{\alpha\beta}(q)$ that satisfies the QYBE $R_{12}R_{13}R_{23}=R_{23}R_{13}R_{12}$, and such that $R_{\alpha\beta}^{\text{T}}=R_{\beta\alpha}$,  relations (\ref{H}) and (\ref{H1}) are equivalent and the relations (\ref{21-43})--(\ref{H1}) satisfy quantum Jacobi property.
\end{theorem}

\proof This is a consequence of the fact that this Poisson algebra comes from the Fock--Rosly one. However we would like to prove it directly at least in one case for sake of completeness: the case in which two entries come from the same monodromy datum  and the third one comes from another monodromy datum. We begin with
$$
\sheet1{M_{21}}R^{-\text{T}}_{12}\sheet2{M_{21}}R_{13}R_{23}\sheet3{M_{43}}
$$
We then have the following chain of equalities in which we use the basic QYBE $R_{12}R_{13}R_{23}=R_{23}R_{13}R_{12}$ and its derivatives:
$$
R^{\text{T}}_{12}R_{23}R_{13}:=R_{21}R_{23}R_{13}=R_{13}R_{23}R_{21}:=R_{13}R_{23}R^{\text{T}}_{12}
$$
and
$$
R_{13}R^{\text{-T}}_{12}R^{-1}_{23}=R^{-1}_{23}R^{\text{-T}}_{12}R_{13}
$$
(we underline the terms in which we use commutation relations or QYBE); note that every $\sheet{i}M$ commutes with $R_{jk}$ if $i\ne\{j,k\}$:
\begin{align}
&\underline{\sheet1{M_{21}}R^{-\text{T}}_{12}\sheet2{M_{21}}}R_{13}R_{23}\sheet3{M_{43}}
=R^{-\text{T}}_{12}\sheet2{M_{21}}R_{12}^{-1}\sheet1{M_{21}}\underline{R_{12}R_{13}R_{23}}\sheet3{M_{43}}\nonumber\\
=&R^{-\text{T}}_{12}\sheet2{M_{21}}R_{12}^{-1}\sheet1{M_{21}}R_{23}R_{13}R_{12}\sheet3{M_{43}}
=R^{-\text{T}}_{12}\sheet2{M_{21}}R_{12}^{-1}R_{23}\underline{\sheet1{M_{21}}R_{13}\sheet3{M_{43}}}R_{12}\nonumber\\
=&R^{-\text{T}}_{12}\sheet2{M_{21}}\underline{R_{12}^{-1}R_{23}R_{13}}\sheet3{M_{43}}R_{13}^{-1}\sheet1{M_{21}}R_{13}R_{12}
=R^{-\text{T}}_{12}\sheet2{M_{21}}R_{13}R_{23}R_{12}^{-1}\sheet3{M_{43}}R_{13}^{-1}\sheet1{M_{21}}R_{13}R_{12}\nonumber\\
=&R^{-\text{T}}_{12}R_{13}\underline{\sheet2{M_{21}}R_{23}\sheet3{M_{43}}}R_{12}^{-1}R_{13}^{-1}\sheet1{M_{21}}R_{13}R_{12}
=\underline{R^{-\text{T}}_{12}R_{13}R_{23}}\sheet3{M_{43}}R_{23}^{-1}\sheet2{M_{21}}\underline{R_{23}R_{12}^{-1}R_{13}^{-1}}\sheet1{M_{21}}R_{13}R_{12}\nonumber\\
=&R_{23}R_{13}R^{-\text{T}}_{12}\sheet3{M_{43}}R_{23}^{-1}R_{13}^{-1}\underline{\sheet2{M_{21}}R_{12}^{-1}\sheet1{M_{21}}}\underline{R_{23}R_{13}R_{12}}
=R_{23}R_{13}R^{-\text{T}}_{12}\sheet3{M_{43}}\underline{R_{23}^{-1}R_{13}^{-1}R_{12}^{\text{T}}}\sheet1{M_{21}}R_{12}^{-\text{T}}\sheet2{M_{21}}R_{13}R_{23}\nonumber\\
=&R_{23}\underline{R_{13}\sheet3{M_{43}}R_{13}^{-1}\sheet1{M_{21}}}R_{23}^{-1}R_{12}^{-\text{T}}\sheet2{M_{21}}R_{13}R_{23}
=R_{23}\sheet1{M_{21}}R_{13}\sheet3{M_{43}}\underline{R_{13}^{-1}R_{23}^{-1}R_{12}^{-\text{T}}}\sheet2{M_{21}}R_{13}R_{23}\nonumber\\
=&R_{23}\sheet1{M_{21}}R_{13}R_{12}^{-\text{T}}\underline{\sheet3{M_{43}}R_{23}^{-1}\sheet2{M_{21}}R_{23}}
=R_{23}\sheet1{M_{21}}\underline{R_{13}R_{12}^{-\text{T}}R_{23}^{-1}}\sheet2{M_{21}}R_{23}\sheet3{M_{43}}\nonumber\\
=&\sheet1{M_{21}}R_{12}^{-\text{T}}\sheet2{M_{21}}R_{13}R_{23}\sheet3{M_{43}}.\nonumber
\end{align}

We present one more calculation demonstrating Jacobi property for the same matrix $M$. We use just one form (\ref{H}) of the commutation relation. We begin with the same expression
$$
\sheet1MR_{12}^{-\text{T}} \sheet2M R_{13}^{-\text{T}}R_{23}^{-\text{T}}\sheet3M=\sheet1MR_{12}^{-\text{T}} R_{13}^{-\text{T}}\sheet2M R_{23}^{-\text{T}}\sheet3M
$$
We first transform the left-hand side:
\begin{align}
&\underline{\sheet1MR_{12}^{-\text{T}} \sheet2M} R_{13}^{-\text{T}}R_{23}^{-\text{T}}\sheet3M=
R_{12}^{-\text{T}}\sheet2MR_{12}^{-1} \sheet1M \underline{R_{12}R_{13}^{-\text{T}}R_{23}^{-\text{T}}}\sheet3M\nonumber\\
=&R_{12}^{-\text{T}}\sheet2MR_{12}^{-1}R_{23}^{-\text{T}} \underline{\sheet1M R_{13}^{-\text{T}}\sheet3M} R_{12}
=R_{12}^{-\text{T}}\sheet2M \underline{R_{12}^{-1}R_{23}^{-\text{T}}R_{13}^{-\text{T}}} \sheet3M R_{13}^{-1}\sheet1M R_{13}R_{12}
\nonumber\\
=&R_{12}^{-\text{T}}R_{13}^{-\text{T}}\underline{\sheet2M R_{23}^{-\text{T}} \sheet3M} R_{12}^{-1}R_{13}^{-\text{T}}\sheet1M R_{13}R_{12}
=R_{12}^{-\text{T}}R_{13}^{-\text{T}}R_{23}^{-\text{T}}\sheet3M R_{23}^{-1} \sheet2M \underline{R_{23} R_{12}^{-1}R_{13}^{-1}}\sheet1M R_{13}R_{12}
\nonumber\\
=&R_{12}^{-\text{T}}R_{13}^{-\text{T}}R_{23}^{-\text{T}}\sheet3M R_{23}^{-1} \sheet2M R_{13}^{-1} R_{12}^{-1}\sheet1M R_{23}R_{13}R_{12}
\nonumber
\end{align}
We now turn to the right-hand side:
\begin{align}
&\sheet1MR_{12}^{-\text{T}} R_{13}^{-\text{T}}\underline{\sheet2M R_{23}^{-\text{T}}\sheet3M}
=\sheet1M\underline{R_{12}^{-\text{T}} R_{13}^{-\text{T}}R_{23}^{-\text{T}}}\sheet3M R_{23}^{-1}\sheet2M R_{23}
\nonumber\\
=&R_{23}^{-\text{T}} \underline{\sheet1M  R_{13}^{-\text{T}}\sheet3M} R_{12}^{-\text{T}}R_{23}^{-1}\sheet2M R_{23}
=R_{23}^{-\text{T}} R_{13}^{-\text{T}}\sheet3M  R_{13}^{-1}\sheet1M \underline{R_{13}R_{12}^{-\text{T}}R_{23}^{-1}}\sheet2M R_{23}
\nonumber\\
=&R_{23}^{-\text{T}} R_{13}^{-\text{T}}\sheet3M  R_{13}^{-1}R_{23}^{-1}\underline{\sheet1M R_{12}^{-\text{T}}\sheet2M} R_{13}R_{23}
=R_{23}^{-\text{T}} R_{13}^{-\text{T}}\sheet3M \underline{R_{13}^{-1}R_{23}^{-1}R_{12}^{-\text{T}}}\sheet2M R_{12}^{-1}\sheet1M R_{12}R_{13}R_{23}
\nonumber\\
=&\underline{R_{23}^{-\text{T}} R_{13}^{-\text{T}}R_{12}^{-\text{T}}}\sheet3M R_{23}^{-1}\sheet2M R_{13}^{-1}R_{12}^{-1}\sheet1M \underline{R_{12}R_{13}R_{23}}
= R_{12}^{-\text{T}}R_{13}^{-\text{T}}R_{23}^{-\text{T}}\sheet3M R_{23}^{-1}\sheet2M R_{13}^{-1}R_{12}^{-1}\sheet1M R_{23}R_{13}R_{12},
\nonumber
\end{align}
which coincides with the final expression in transformations of the left-hand side. \endproof

\subsection{Decorated character variety}

In order to define a Poisson structure on the decorated character variety
$$
\mathcal M_{g,s,n}^k:=\mathcal R_{g,s,n}^k\slash_{\prod_{j=1}^n U_j},
$$
we prove that the quotient by unipotent Borel sub-groups is a Poisson reduction.

In the $SL_2(\mathbb C)$ case, we have the following Poisson reduction for monodromy data corresponding to passing along the hole boundary: if $M$ corresponds to a path along the boundary that goes clockwise (the hole is to the left w.r.t. the path direction), then $m_{22}=0$, i.e., the lower right element vanishes.

For generic $k\geq 2$, we have the following

\begin{lemma}\label{Poisson}
Consider the monodromy data  $M\in SL_k(\mathbb C)$ corresponding to paths that go clockwise along boundaries of holes (they may start and terminate at the same cusp if a hole has only one cusp). The reduction
\beq
\label{starstar}
m_{i,j}=0,\ i+j\ge k+2,
\eeq
i.e., all its entries below the main anti-diagonal vanish, is a Poisson reduction.
\end{lemma}

\proof The proof is based on the following observation: if $M$ is a matrix corresponding to a path that is leftmost at the starting cusp and
rightmost at the terminating cusp (examples are $M_2$ in the basic relation (\ref{basic1}) and $M_{41}$ in (\ref{32-41})), then (Poisson or quantum) commutation relations of elements $m_{i,j}$ of this matrix with elements of every other matrix or among themselves are such that every
term of the corresponding commutation relation necessarily contains an element $m_{k,j}$ or $m_{i,l}$ with $k\ge i$ and $l\ge j$, i.e., 
$k+j\ge i+j$ and $i+l\ge i+j$. Therefore imposing a constraint $m_{i,j}=0$ for $i+j>k+1$ is consistent: commutation relations of such elements with all other elements of algebra automatically vanish.
\endproof

\begin{remark}
From a purely algebraic standpoint, one may consider other Poisson reductions; the one for which $m_{i,j}=0$ for elements below the main antidiagonal is consistent with factoring out a gauge freedom associated with Borel subgroups at cusps.  
\end{remark}

The Poisson reduction in Lemma \ref{Poisson} is the quotient w.r.t. unipotent Borel subgroups $U_i\subset SL_k(\mathbb C)$ associated with the cusps, therefore it endows the decorated character variety  $\mathcal M_{g,s,n}^k$ with a Poisson structure.

\subsection{Powers of matrices}
Using the commutation relations (\ref{21-43})--(\ref{H1}) in the $R$-matrix form, we obtain the following generalizations of these relations to powers of matrices:
\begin{align}
&\sheet{1}{M_{21}^{\tcr{p}}} R_{12}\sheet{2} {M_{43}^{\tcb{m}}}R_{12}^{-1} =R_{12} \sheet{2}{M_{43}^{\tcb{m}}}R^{-1}_{12} \sheet{1}{M_{21}^{\tcr{p}}}, \quad {\tcr{p}},{\tcb{m}}\in \mathbb Z;
\label{21-43-nm}
\\
&\sheet{1}{M_{41}} R_{12}^{-\text{T}}\sheet{2} {M_{32}^{\tcr{p}}}R_{12}^{\text{T}} =R_{12} \sheet{2}{M_{32}^{\tcr{p}}}R^{-1}_{12} \sheet{1}{M_{41}},\quad {\tcr{p}}\in \mathbb Z
\label{32-41-nm}
\\
&\sheet{1}{M_{31}} R_{12}^{-\text{T}}\sheet{2} M_{42}R_{12}^{-1} =R_{12} \sheet{2}{M_{42}}R^{-1}_{12} \sheet{1}{M_{31}}\ \hbox{\tcr{no generalization}};
\label{31-42-nm}
\\
&R_{12}^{\text{T}}\sheet{1} {M_{ij}^{\tcr{p}}} R_{12}^{-\text{T}}\sheet{2}{M_{ij}}=\sheet{2}{M_{ij}} R_{12}^{-1} \sheet{1}{M_{ij}^{\tcr{p}}} R_{12},\quad {\tcr{p}}\in\mathbb Z;
\label{H-nm}
\\
&\sheet{1} {M_{ij}} R_{12}^{-\text{T}}\sheet{2}{M_{ij}^{\tcr{p}}}R_{12}^{\text{T}}= R_{12}\sheet{2}{M_{ij}^{\tcr{p}}} R_{12}^{-1} \sheet{1}{M_{ij}},\quad {\tcr{p}}\in \mathbb Z.
\label{H1-nm}
\end{align}
Note that for the {\em same} matrix $M_{ij}$ we can take powers of only one of the matrices $M_{ij}$ in relations (\ref{H}) and (\ref{H1}) but not powers of both matrices.

\subsection{Semiclassical limit}

By taking $q=\exp(-i\pi\hbar)$, we can expand the Kulish--Sklyanin matrix $R_{12}$ as:
$$
R_{12}(q)=(1+ \frac{i\pi \hbar}{2}+\mathcal O(\hbar^2))\sum_{i,j}\sheet{1}{e_{ii}}\otimes \sheet{2}{e_{jj}}
+(-i\pi \hbar+\mathcal O(\hbar^3)) \sum_i \sheet{1}{e_{ii}}\otimes \sheet{2}{e_{ii}}+(-2 i\pi \hbar+\mathcal O(\hbar^2))\sum_{j>i}
\sheet{1}{e_{ij}}\otimes \sheet{2}{e_{ji}}
$$
so that we obtain:
$$
R_{12}(q)=\sheet{1}{1}\otimes\sheet{2}{1}+ i\pi\hbar r + \mathcal O(\hbar^2),\qquad R_{12}(1/q)=\sheet{1}{1}\otimes\sheet{2}{1}- i\pi\hbar r + \mathcal O(\hbar^2),
$$
where:
\be\label{eq:semi-cl}
r= \frac{1}{2}\sum_{i,j}\sheet{1}{e_{ii}}\otimes \sheet{2}{e_{jj}}- \sum_i \sheet{1}{e_{ii}}\otimes \sheet{2}{e_{ii}}-2
\sum_{j>i} \sheet{1}{e_{ij}}\otimes \sheet{2}{e_{ji}}.
\eeq
Now, using the correspondence principle that $[A^\hbar,B^\hbar]\mapsto i\pi \hbar \{A,B\}$, we can take the semiclassical limits of  (\ref{21-43}):
$$
\sheet{1}{M_{21}}\sheet{2}{M_{43}} + i\pi\hbar \sheet{1}{M_{21}}r \sheet{2}{M_{43}} - i\pi\hbar\sheet{1}{M_{21}}\sheet{2}{M_{43}} r =
\sheet{2}{M_{43}} \sheet{1}{M_{21}}+
 i\pi\hbar r \sheet{2}{M_{43}} \sheet{1}{M_{21}}- i\pi\hbar \sheet{2}{M_{43} }r\sheet{1}{M_{21}},
$$
so that  (\ref{21-43}) becomes
\be\label{21-43sc}
\{\sheet{1}{M_{21}},\sheet{2}{M_{43}}\}=- \sheet{1}{M_{21}}r \sheet{2}{M_{43}} +\sheet{1}{M_{21}}\sheet{2}{M_{43}} r+
 r \sheet{2}{M_{43}} \sheet{1}{M_{21}}- \sheet{2}{M_{43} }r\sheet{1}{M_{21}},
\ee
In the same way (\ref{32-41})  becomes:
\be\label{32-41sc}
\{\sheet{1}{M_{41}},\sheet{2}{M_{32}}\}= \sheet{1}{M_{41}}r^T \sheet{2}{M_{32}} -\sheet{1}{M_{41}}\sheet{2}{M_{32}} r^T+
 r \sheet{2}{M_{32}} \sheet{1}{M_{41}}- \sheet{2}{M_{32} }r\sheet{1}{M_{41}},
\ee
while (\ref{31-42}) becomes:
\be\label{31-42sc}
\{\sheet{1}{M_{31}},\sheet{2}{M_{42}}\}= \sheet{1}{M_{31}}r^T \sheet{2}{M_{42}} +\sheet{1}{M_{31}}\sheet{2}{M_{42}} r+
 r \sheet{2}{M_{42}} \sheet{1}{M_{31}}- \sheet{2}{M_{42} }r\sheet{1}{M_{31}},
\ee
and (\ref{H}) becomes:
\be\label{Hsc}
\{\sheet{1}{M_{ij}},\sheet{2}{M_{ij}}\}= \sheet{1}{M_{ij}}r^T \sheet{2}{M_{ij}} -r^T \sheet{1}{M_{ij}}\sheet{2}{M_{ij}} +
  \sheet{2}{M_{ij}} \sheet{1}{M_{ij}} r- \sheet{2}{M_{ij} }r\sheet{1}{M_{ij}}.
\ee
We let $m^{\alpha,\beta}_{k,l}$ denote the $(k,l)$-element of the matrix $M_{\alpha,\beta}$. For matrix elements, we have the following Poisson relations (in the formulas below, double indices imply summations):
\begin{align}
&\{m^{2,1}_{i,j},m^{4,3}_{k,l}\}=m^{2,1}_{i,s}m^{4,3}_{s,l}\delta_{j,k}\theta(j{-}s)-m^{2,1}_{i,l}m^{4,3}_{k,j}\theta(j{-}l)+m^{2,1}_{s,j}m^{4,3}_{k,s}\delta_{i,l}\theta(s{-}i)-m^{2,1}_{k,j}m^{4,3}_{i,l}\theta(k{-}i),
\label{21-43-Pb}
\\
&\{m^{4,1}_{i,j},m^{3,2}_{k,l}\}=-m^{4,1}_{i,s}m^{3,2}_{s,l}\delta_{j,k}\theta(s{-}j)+m^{4,1}_{i,l}m^{3,2}_{k,j}\theta(l{-}j)+m^{4,1}_{s,j}m^{3,2}_{k,s}\delta_{i,l}\theta(s{-}i)-m^{4,1}_{k,j}m^{3,2}_{i,l}\theta(k{-}i),
\label{32-41-Pb}
\\
&\{m^{3,1}_{i,j},m^{4,2}_{k,l}\}=-m^{3,1}_{i,s}m^{4,2}_{s,l}\delta_{j,k}\theta(s{-}j)-m^{3,1}_{i,l}m^{4,2}_{k,j}\theta(j{-}l)+m^{3,1}_{s,j}m^{4,2}_{k,s}\delta_{i,l}\theta(s{-}i)-m^{3,1}_{k,j}m^{4,2}_{i,l}\theta(k{-}i) {+2 m^{3,1}_{i,s}m^{4,2}_{s,l}},
\label{31-42-Pb}
\\
&\{m^{\alpha,\beta}_{i,j},m^{\alpha,\beta}_{k,l}\}=-m^{\alpha,\beta}_{i,s}m^{\alpha,\beta}_{s,l}\delta_{j,k}\theta(s{-}j)+m^{\alpha,\beta}_{s,j}m^{\alpha,\beta}_{k,s}\delta_{i,l}\theta(s{-}i) +m^{\alpha,\beta}_{k,j}m^{\alpha,\beta}_{i,l}(\theta(l{-}j)-\theta(k{-}i)),
\label{H-Pb}
\end{align}
{where $\theta(k):={\rm sign}(k)+1$.}
\begin{remark}
Note that quantum commutation relations (\ref{H}) and (\ref{H1}) have the same semiclassical limit (\ref{H-Pb}).
\end{remark}

For a monodromy datum $M_2^1$ corresponding to an arc starting and terminating at different cusps, we have
\beq
\label{3.100}
\{m_{2_{i,j}}^1,m^1_{2_{k,l}}\}=m_{2_{i,j}}^1m^1_{2_{k,l}}(\theta(i-k)-\theta(j-l)).
\eeq

\subsection{Casimirs of the Poisson algebra of monodromy data}

We now address the problem of constructing Casimirs for the Poisson algebra of monodromy data.
For  technical convenience, we do not impose the restriction $\det M=1$, although relations (\ref{21-43-Pb})--(\ref{H-Pb}) imply that
determinants of all monodromy data corresponding to arcs starting and terminating at the same cusp are central.

\begin{theorem}
\label{thm:traces}
In any system of $SL_k(\mathbb C)$ monodromy data, for a monodromy datum  $M$ corresponding to an arc homeomorphic to circumnavigating a single hole without cusps, all elements $\tr [M^p]$, $p=1,\dots,k$, are Casimirs of the Poisson algebra of monodromy data.
\end{theorem}

\proof The monodromy datum  $M$ corresponding to an arc homeomorphic to circumnavigating a single hole without cusps can be identified with the matrix $M_{21}$, while all other monodromy data can be identified with a matrix denoted by $M_{3,[x]}$ that corresponds to a path starting at the same cusp as $M_{21}$ (to the right of both ends of $M_{21}$) and terminating at a different cusp (denoted $[x]$).  In the quantum case, we have the relation
\beq
\label{21-3x}
\sheet{1}{M_{21}} \sheet{2} {M_{3,[x]}} = \sheet{2}{M_{3,[x]}}R^{-1}_{12} \sheet{1}{M_{21}} R_{12},
\eeq
which admits an immediate generalization to powers of $M_{12}$:
\beq
\label{21-3x-nm}
\sheet{1}{M_{21}^p} \sheet{2} {M_{3,[x]}} = \sheet{2}{M_{3,[x]}}R^{-1}_{12} \sheet{1}{M_{21}^p} R_{12},
\quad p\in \mathbb Z;
\eeq
Evaluating traces in space $1$ in the  semiclassical limits of relations
(\ref{21-43-nm}), (\ref{32-41-nm}), (\ref{H-nm}), and (\ref{21-3x-nm}), we find that the traces of the matrix $M_{21}^n$ (or $M_{32}^n$)  Poisson commute with
elements of all other matrices (and elements of the matrix $M_{21}$ itself), which completes the proof. We do this computation in detail only for relations
(\ref{21-3x-nm}), all other cases are completely analogous.

The semiclassical limit of relation (\ref{21-3x-nm}) reads:
\beq
\{m_{21_{i,j}}^p m_{{3,[x]}_{k,l}}\}=-\sum_{s=1}^n [m_{21_{s,j}}^p m_{{3,[x]}_{k,s}}\delta_{i,l}\theta(s-i)+m_{21_{i,l}}^p m_{{3,[x]}_{k,j}}\theta(j-l)].
\label{21-3x-Pb}
\eeq
Taking the sum over $i$ with $i=j$, we obtain
\begin{align*}
\sum_{i=1}^n \{m_{21_{i,i}}^p m_{{3,[x]}_{k,l}}\}  &=-\sum_{s=1}^n \sum_{i=1}^n  [m_{21_{s,i}}^p m_{{3,[x]}_{k,s}}\delta_{i,l}\theta(s-i)+m_{21_{i,l}}^p m_{{3,[x]}_{k,i}}\theta(i-l)]\\
   & =-\sum_{s=1}^n  m_{21_{s,l}}^p  m_{{3,[x]}_{k,s}}\theta(s-l)
+\sum_{i=1}^n  m_{21_{i,l}}^p m_{{3,[x]}_{k,i}}\theta(i-l)=0.
\end{align*}

Therefore, $\tr M^k$ are Casimirs for the algebra of elements of {\it any} matrix $M$ corresponding to a path that starts and terminates at the
same cusp.  \endproof

In Section \ref{s:examples}, we address the question of the dimension of the symplectic leaves in some examples.

\subsection{Reduction to the $SL_2$ decorated character variety}
Let us select the same cusped lamination as in \cite{ChM1}, then for every arc
$\mathfrak a$ in the lamination we associate a matrix $M_{\mathfrak a}\in SL_k$. This gives $6g-6+3s+2n$ matrices in $SL_k$.

It is easy to prove that the following character
$$
\begin{array}{ll}\tr_K:& SL_k(\mathbb C) \to \mathbb C\\
& M\mapsto \Tr(MK),\\
\end{array}\qquad \hbox{where}\quad K=\left(
\begin{array}{cccc}
0  &  \dots &0 & 0    \\
\dots  &  \dots &\dots& \dots    \\
0  &  \dots &0 & 0    \\
-1 &   0   &0 & 0 \\
\end{array}
\right),
$$
is well defined on $M_{g,s,n}^k$. Recall that $\tr_K(M)= - m_{1k}$.

Define the $\lambda$-length of the arc ${\mathfrak a}$ by $\tr_K(M_{\mathfrak a})$ and introduce the following operation:
$$
\sheet{12}{\tr}_K(\sheet{1}M_1\sheet{2}M_2):= \sheet{12}{\tr}(\sheet{1}M_1\sheet{1} K \sheet{2}M_2 \sheet{2} K)
$$
By taking $\sheet{12}{\tr}_K$ in all relations (\ref{21-43}), (\ref{32-41}), (\ref{31-42}) and (\ref{H}) we obtain the $\lambda$-lengths algebra on the $SL_2$ decorated character variety found in \cite{ChM1}.

\section{Examples of  algebras of monodromy data}
\label{s:examples}

\subsection{Case of only one monodromy datum }

{In the case when $2 g-2+s+n=1$ we only have one monodromy datum. We have two different situations, the first is when $M$ comes back to the same cusp, the second when it connects different cusps. In the former case the following lemma holds true:}

\begin{lemma}
\label{lm:M12-Casimirs}
For a general-position monodromy datum subject to algebra (\ref{H-Pb}), the maximum dimension of the Poisson leaves is $k(k-1)$. The only
Casimirs in this case are $\tr [M^p]$ with $p=1,\dots,k$.
\end{lemma}

\proof
The {proof} is based on the following observation. Let the classical $M$ has a diagonal form  $m_{i,j}=\delta_{i,j}\lambda_i$ with all $\lambda_i$ distinct and nonzero. The Poisson brackets are then nonzero only inside the pairs $(m_{i,j},m_{j,i})$ with $1\le i<j\le k$, for which we have
$$
\{m_{i,j},m_{j,i}\}=\lambda_{j}(\lambda_{j}-\lambda_{i}),
$$
so these brackets are non-degenerate. The minimum Poisson dimension of the corresponding leaf is thus $k^2-k$, and it is simultaneously the maximum possible Poisson dimension as the traces $\tr M^p$ by Theorem~\ref{thm:traces} are $k$ algebraically independent Casimirs of the algebra (\ref{H-Pb}). \endproof

We next address the problem of Casimirs for the case when the monodromy datum corresponds to an arc that connects different cusps -  algebra (\ref{3.100}):

\begin{lemma}
\label{lm:M1x-Casimirs}
\cite{FM, ChMaz}
The central elements of the algebra (\ref{3.100}) in the case of nonrestricted matrices $M$ are ratios $M_d^{\text{UL}}/M_{k-d}^{\text{LR}}$
of upper-left and lower-right minors of the respective dimensions $d\times d$ and $(k-d)\times (k-d)$ for $d=1,\dots,k$.
\end{lemma}

We now derive analogues of Lemmas \ref{lm:M12-Casimirs} and \ref{lm:M1x-Casimirs} in the presence of  the constraints (\ref{starstar}).

\begin{lemma}
\label{lm:M12-Casimirs-UT}
For a monodromy datum $M$ subject to the algebra (\ref{H-Pb}) with restriction (\ref{starstar}) imposed, the maximum dimension of Poisson leaves
is $k(k-1)/2-[k/2]$ and the $k$ Casimirs are $\tr [M^p]$, $p=1,\dots,k$, as in the nonrestricted case, plus $[k/2]$ Casimirs defined by
\beq
\label{Ci}
C_i:=m_{i,k+1-i}/m_{k+1-i,i},\quad i=1,\dots,[k/2].
\eeq
\end{lemma}

\begin{proof}
Traces are Casimirs for the general matrix $M$ and they remain Casimirs for any Poisson reduction. Anti-diagonal elements of $M$
have homogeneous Poisson relations with all other elements:
$$
\{m_{i,k+1-i},m_{k,l}\}=m_{i,k+1-i}m_{k,l}\bigl[-\delta_{k+1-i,k}+\delta_{k+1-i,l}+\delta_{i,l}-\delta_{i,k}\bigr],
$$
so the ratios (\ref{Ci}) have zero Poisson brackets with all $m_{k,l}$.
It remains to prove that the highest Poisson dimension matches the number of already found Casimirs. In order to prove it, take the reduced matrix $M$ in the form in which the diagonal elements $m_{i,i}$ with $1\le i\le [k/2]$ and all anti-diagonal elements $m_{i,k+1-i}$, $i=1,\dots,k$, are nonzero and are not algebraically related. Then it is a direct calculation to check that the commutation relations
are closed inside quadruples $\bigl(m_{i,j},m_{k+1-i,j},m_{j,i},m_{k+1-j,i}\bigr)$ with $1\le j<i\le [k/2]$, doubles $(m_{(k+1)/2,j}, m_{j,(k+1)/2})$ with $1\le j<(k+1)/2$, triples $\bigl(m_{i,i},m_{k+1-i,i},m_{i,k+1-i}\bigr)$ and singles $m_{(k+1)/2,(k+1)/2}$. (Of course,
doubles and singles occur only for odd $n$.) It is then a straightforward calculation to show that quadruples and doubles have full Poisson
dimension whereas triples and singles have zero Poisson dimensions, so all their elements correspond to Casimirs. But the total number of
elements in triples and singles (if any) is exactly $k+[k/2]$, i.e., the number of Casimirs listed above.
\end{proof}

\begin{lemma}
\label{lm:M1x-Casimirs-res}
\cite{FM, ChMaz}
The central elements of the algebra (\ref{3.100}) in the case of matrices $M$ with restrictions (\ref{starstar}) are
\beq
\label{Cd}
\hat C_d=\frac{M_d^{\text{UL}}\prod_{i=1}^d \bigl[m_{i,k+1-i} m_{k+1-i,i}\bigr]}{M_{k-d}^{\text{UL}}},\quad d=0,\dots,\Bigl[\frac{k-1}{2}\Bigr].
\eeq
(See Fig.~\ref{fig:M3x}.)
In this formula, both minors are upper-left, of sizes $d\times d$ and $(k-d)\times(k-d)$. We have $[(k+1)/2]$ such Casimirs and the maximum Poisson dimension is $k(k+1)/2-[(k+1)/2]$.
\end{lemma}

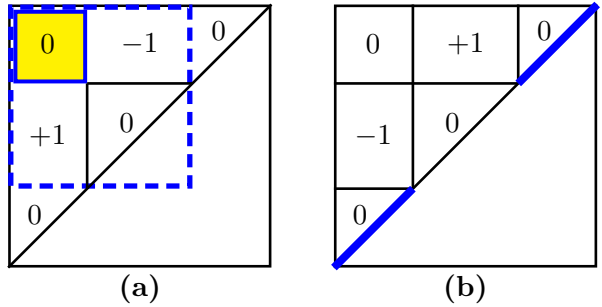
\begin{figure}[tb]
{\psset{unit=0.7}
\begin{pspicture}(-3,-3)(3,3)
\psframe[linecolor=black, linewidth=1pt, fillstyle=solid, fillcolor=white](-2.5,2.5)(2.5,-2.5)
\psframe[linecolor=blue, linestyle=dashed, linewidth=2pt, fillstyle=solid, fillcolor=white](-2.5,2.5)(1,-1)
\psframe[linecolor=blue, linewidth=1.5pt, fillstyle=solid, fillcolor=yellow](-2.4,2.4)(-1,1)
\pcline[linewidth=1pt](-2.5,-2.5)(2.5,2.5)
\pcline[linewidth=1pt](-1,1)(-1,-1)
\pcline[linewidth=1pt](-1,1)(1,1)
\put(-1.75,1.75){\makebox(0,0)[cc]{\hbox{{$0$}}}}
\put(-1.75,0){\makebox(0,0)[cc]{\hbox{{$+1$}}}}
\put(-2,-1.5){\makebox(0,0)[cc]{\hbox{{$0$}}}}
\put(-0.25,0.25){\makebox(0,0)[cc]{\hbox{{$0$}}}}
\put(1.5,2){\makebox(0,0)[cc]{\hbox{{$0$}}}}
\put(0,1.75){\makebox(0,0)[cc]{\hbox{{$-1$}}}}
\put(0,-2.9){\makebox(0,0)[cc]{\hbox{{\bf (a)}}}}
%
\end{pspicture}
\begin{pspicture}(-3,-3)(3,3)
\psframe[linecolor=black, linewidth=1pt, fillstyle=solid, fillcolor=white](-2.5,2.5)(2.5,-2.5)
\pcline[linewidth=1pt](-2.5,-2.5)(2.5,2.5)
\pcline[linewidth=1pt](-1,2.5)(-1,-1)
\pcline[linewidth=1pt](-2.5,1)(1,1)
\pcline[linewidth=1pt](1,2.5)(1,1)
\pcline[linewidth=1pt](-2.5,-1)(-1,-1)
\pcline[linewidth=3pt, linecolor=blue](-2.5,-2.5)(-1,-1)
\pcline[linewidth=3pt, linecolor=blue](2.5,2.5)(1,1)
\put(-1.75,1.75){\makebox(0,0)[cc]{\hbox{{$0$}}}}
\put(-1.75,0){\makebox(0,0)[cc]{\hbox{{$-1$}}}}
\put(-2,-1.5){\makebox(0,0)[cc]{\hbox{{$0$}}}}
\put(-0.25,0.25){\makebox(0,0)[cc]{\hbox{{$0$}}}}
\put(1.5,2){\makebox(0,0)[cc]{\hbox{{$0$}}}}
\put(0,1.75){\makebox(0,0)[cc]{\hbox{{$+1$}}}}
\put(0,-2.9){\makebox(0,0)[cc]{\hbox{{\bf (b)}}}}
%
\end{pspicture}
}
\caption{Constructing Casimirs for the restricted matrix $M$ subject to Poisson algebra (\ref{3.100}). Numbers in
the corresponding rectangles or triangles indicate the sign of homogeneous commutation relations between elements in
the corresponding region and (a) the ratios of minors $M_d^{\text{UL}}/M_{k-d}^{\text{UL}}$ ($d<k-d$) and (b) the products
$\prod_{i=1}^d [m_{i,k+1-i}m_{k+1-i,i}]$ of elements on the antidiagonal.
We see that these signs are complementary in all regions.}
\label{fig:M3x}
\end{figure}

\proof As illustrated in Fig.~\ref{fig:M3x}, all matrix elements have homogeneous commutation relations with any minor
$M_d^{\text{UL}}$, for the ratio $M_d^{\text{UL}}/M_{k-d}^{\text{UL}}$ the coefficients are $+1,0,-1$ depending on the region which
this matrix element belongs to. They are depicted in Fig.~\ref{fig:M3x}(a). Next, all matrix elements have homogeneous commutation relations with any element $m_{i,k+1-i}$ on the main antidiagonal. For the product $\prod_{i=1}^d \bigl[m_{i,k+1-i} m_{k+1-i,i}\bigr]$ of these elements, the corresponding coefficients are depicted in Fig.~\ref{fig:M3x}(b). We see that the two patterns are complementary, so the product in the
right-hand side of (\ref{Cd}) commutes with all matrix elements.

If we again assume that only anti-diagonal elements and upper half of diagonal elements of $M$ are nonzero, then, as in the proof of Lemma~\ref{lm:M12-Casimirs-UT}, we can split the whole set of elements into quadruples, doubles, triples, and singles; as in the above
proof, all quadruples and doubles will then have the full Poisson dimension, but in contrast to the proof of Lemma~\ref{lm:M12-Casimirs-UT},
the Poisson dimension of triples will be two, not zero (and it obviously remains zero for singles). Thus we are losing exactly $2[k/2]$ central
elements as compared to the previous case, and the maximum Poisson dimension in this case is $k(k+1)/2-[(k+1)/2]$. \endproof

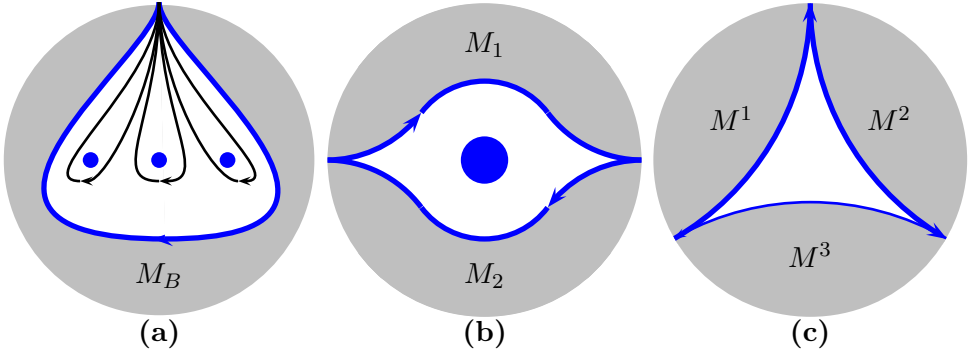
\begin{figure}[tb]
{\psset{unit=0.7}
\begin{pspicture}(-3,-3)(3,3)
\pscircle[linecolor=white, linewidth=1pt,fillstyle=solid, fillcolor=lightgray](0,0){3}
\psbezier[linecolor=blue, linewidth=2pt, fillstyle=solid,fillcolor=white](0,3)(0,2)(-5,-1.5)(0.12,-1.5)
\psbezier[linecolor=blue, linewidth=2pt, fillstyle=solid,fillcolor=white]{->}(0,3)(0,2)(5,-1.5)(-0.12,-1.5)
\pscircle[linecolor=white, linewidth=1pt,fillstyle=solid, fillcolor=blue](0,0){0.2}
\pscircle[linecolor=white, linewidth=1pt,fillstyle=solid, fillcolor=blue](-1.3,0){0.2}
\pscircle[linecolor=white, linewidth=1pt,fillstyle=solid, fillcolor=blue](1.3,0){0.2}
\psbezier[linewidth=1pt](0,3)(0,1.5)(1,-0.4)(1.5,-0.4)
\psbezier[linewidth=1pt]{->}(0,3)(0,1.5)(2.5,-0.4)(1.5,-0.4)
\psbezier[linewidth=1pt](0,3)(0,1.5)(-2.5,-0.4)(-1.5,-0.4)
\psbezier[linewidth=1pt]{->}(0,3)(0,1.5)(-1,-0.4)(-1.5,-0.4)
\psbezier[linewidth=1pt](0,3)(0,.8)(-1,-0.4)(0,-0.4)
\psbezier[linewidth=1pt]{->}(0,3)(0,.8)(1,-0.4)(0,-0.4)
\put(0,-2.2){\makebox(0,0)[cc]{\hbox{{$M_B$}}}}
\put(0,-3.3){\makebox(0,0)[cc]{\hbox{{\bf (a)}}}}
\end{pspicture}
\begin{pspicture}(-3,-3)(3,3)
\psclip{\pscircle[linecolor=white](0,0){3}}
\pswedge[linewidth=1pt,linecolor=lightgray,fillcolor=lightgray, fillstyle=solid](-3,2.25){2.25}{-90}{0}
\pswedge[linewidth=1pt,linecolor=lightgray,fillcolor=lightgray, fillstyle=solid](-3,-2.25){2.25}{0}{90}
\pswedge[linewidth=1pt,linecolor=lightgray,fillcolor=lightgray, fillstyle=solid](3,2.25){2.25}{180}{270}
\pswedge[linewidth=1pt,linecolor=lightgray,fillcolor=lightgray, fillstyle=solid](3,-2.25){2.25}{90}{180}
\pswedge[linewidth=1pt,linecolor=lightgray,fillcolor=lightgray, fillstyle=solid](0,0){3}{37}{143}
\pswedge[linewidth=1pt,linecolor=lightgray,fillcolor=lightgray, fillstyle=solid](0,0){3}{217}{323}
\pscircle[linecolor=white, linewidth=1pt,fillstyle=solid, fillcolor=white](0,0){1.5}
\pscircle[linecolor=white, linewidth=1pt,fillstyle=solid, fillcolor=blue](0,0){.5}
\psarc[linewidth=2pt,linecolor=blue](0,0){1.5}{37}{143}
\psarc[linewidth=2pt,linecolor=blue](0,0){1.5}{217}{323}
\psarc[linewidth=2pt,linecolor=blue](-3,-2.25){2.25}{37}{90}
\psarc[linewidth=2pt,linecolor=blue](3,2.25){2.25}{217}{270}
\psarc[linewidth=2pt,linecolor=blue]{->}(3,-2.25){2.25}{90}{143}
\psarc[linewidth=2pt,linecolor=blue]{->}(-3,2.25){2.25}{270}{323}
\endpsclip
\put(0,2.2){\makebox(0,0)[cc]{\hbox{{$M_1$}}}}
\put(0,-2.2){\makebox(0,0)[cc]{\hbox{{$M_2$}}}}
\put(0,-3.3){\makebox(0,0)[cc]{\hbox{{\bf (b)}}}}
\end{pspicture}
\begin{pspicture}(-3,-3)(3,3)
\psclip{\pscircle[linecolor=white](0,0){3}}
\pswedge[linewidth=1pt,linecolor=lightgray,fillcolor=lightgray, fillstyle=solid](-5.2,3){5.2}{300}{360}
\pswedge[linewidth=1pt,linecolor=lightgray,fillcolor=lightgray, fillstyle=solid](5.2,3){5.2}{180}{240}
\pswedge[linewidth=1pt,linecolor=lightgray,fillcolor=lightgray, fillstyle=solid](0,-6){5.2}{60}{120}
\psarc[linewidth=2pt,linecolor=blue]{->}(-5.2,3){5.2}{300}{360}
\psarc[linewidth=2pt,linecolor=blue]{->}(5.2,3){5.2}{180}{240}
\psarc[linewidth=1pt,linecolor=blue]{->}(0,-6){5.2}{60}{120}
\endpsclip
\put(-1.5,.8){\makebox(0,0)[cc]{\hbox{{$M^1$}}}}
\put(1.5,.8){\makebox(0,0)[cc]{\hbox{{$M^2$}}}}
\put(0,-1.8){\makebox(0,0)[cc]{\hbox{{$M^3$}}}}
\put(0,-3.3){\makebox(0,0)[cc]{\hbox{{\bf (c)}}}}
\end{pspicture}
}
\caption{Three cases of monodromy data: $\Sigma_{0,s+1,1}$ (a), $\Sigma_{0,2,2}$ (b), and $\Sigma_{0,1,3}$ (c).}
\label{fig:sigma}
\end{figure}

\subsection{Monodromy algebras for $\Sigma_{0,s+1,1}$}
We first address the case of monodromy data for a disc bounded by a hole with a single bordered cusp and with $s$ holes in the interior (Fig.~\ref{fig:sigma}(a)). The basis of monodromy data is constituted by $s$ $(k\times k)$-matrices $M_{2j,2j-1}$,\ $j=1,\dots, s$ where we order linearly all $2s$ ends of cycles corresponding to these monodromy data; the matrix $M_{2j,2j-1}$ corresponds to the path that starts and terminates at the bordered cusp and circumnavigates the $j$th hole inside the disc going clockwise. The quantum and Poisson algebras of elements of these matrices are described by relations
(\ref{21-43}), (\ref{H}), (\ref{H1}) and (\ref{21-43-Pb}), (\ref{H-Pb}).

We introduce also the boundary monodromy datum
\beq
\label{Mb}
M_B:=M_{21}M_{43}\cdots M_{2s-2,2s-3}M_{2s,2s-1}.
\eeq
We begin with the lemma describing a nonrestricted case.

\begin{lemma}
The maximum Poisson dimension of leaves of the $SL_k(\mathbb C)$-algebra of monodromy data for $\Sigma_{0,s+1,1}$ is $sk(k-1)$. The $sk$ central elements are $\tr [M_{2j,2j-1}]^p$, $j=1,\dots,s$, $p=1,\dots,k$.
\end{lemma}
\begin{proof}
That the traces of powers of $M_{2j,2j-1}$ are central was proved in Theorem~\ref{thm:traces}.
Here we prove that the general Poisson dimension of a Poisson leaf is $sk(k-1)$. For this, we again evaluate the rank of the Poisson bi-vector at a specific point in the phase space. It is convenient to take a point at which all monodromy data are diagonal and all their diagonal elements are distinct and nonzero, $m_{{2j,2j-1}_{p,l}}= \delta_{p,l}\lambda^{(j)}_p$, with all $\lambda^{(j)}_p\ne 0$ and such that $\lambda^{(j_1)}_{p_1}=\lambda^{(j_2)}_{p_2}$ only for $p_1=p_2$ and $j_1=j_2$. It is easy to check that nonzero entries in the Poisson bi-vector correspond to the Poisson brackets between $m_{{2j,2j-1}_{p,l}}$
and $m_{{2j,2j-1}_{l,p}}$ with $p\ne l$ and these brackets, evaluated at the chosen point, are
\begin{align}
&\{m_{{2j,2j-1}_{p,l}},m_{{2j,2j-1}_{l,p}}\}=(\lambda^{(j)}_p-\lambda^{(j)}_l)[\lambda^{(j)}_p\theta(p-l)+\lambda^{(j)}_l \theta(l-p)];
\\
&\{m_{{2j,2j-1}_{p,l}},m_{{2j,2j-1}_{l,p}}\}=(\lambda^{(j)}_p-\lambda^{(j)}_l)(\lambda^{(i)}_p-\lambda^{(i)}_l)\theta(l-p),\quad j<i.
\end{align}
So, the Poisson bi-vector at this point has mostly vanishing entries except $2s$ non-vanishing relations for every fixed pair $(p,l)$, with $p\neq l$. We can organise rows and columns by ordering them by the pairs $(p,l)$, starting from $(1,2)$ and ending with $(k-1,k)$. In this way the Poisson bi-vector becomes block diagonal with $2s\times 2 s$ blocks of non-zero determinant  provided all $\lambda^{(j)}_p$ are nonzero and distinct.
There are $k(k-1)/2$ such blocks, which proves that the ranks is  $sk(k-1)$.
\end{proof}

We now consider the actual situation with the Lagrangian restriction (\ref{starstar}) imposed on the monodromy datum  $M_B$ (\ref{Mb}).

\begin{theorem}\label{thm:S0s1}
The maximum Poisson dimension of leaves of the $SL_k(\mathbb C)$-algebra of monodromy data for $\Sigma_{0,s+1,1}$ with the restriction
(\ref{starstar}) imposed on the only boundary monodromy datum  $M_B$ (\ref{Mb}) is $sk(k-1)-k(k-1)/2-[k/2]$. Besides the
standard $sk$ central elements that are $\tr M_{2j,2j-1}^p$, $j=1,\dots,s$, $p=1,\dots,k$ we have $[k/2]$ central elements
having the form (\ref{Ci}) for the matrix $M_B$, i.e., $C_i=m_{B_{i,k+1-i}}/m_{B_{k+1-i,i}}$.
\end{theorem}

\begin{proof}
We first prove that $C_i=[m_B]_{i,k+1-i}/[m_B]_{k+1-i,i}$ are central. Every element $[m_B]_{i,k+1-i}$ has homogeneous Poisson relations
with every matrix element of every matrix $M_{2j,2j-1}$. In Poisson relations we can identify $M_B$ with $M_{41}$ and every $M_{2j,2j-1}$---with $M_{32}$. Actual Poisson brackets coincide with those inside the same matrix $M_{41}$,
$$
\{m_{{41}_{i,k+1-i}},m_{{32}_{p,l}}\}=m_{{41}_{i,k+1-i}}m_{{32}_{p,l}}\bigl[-\delta_{k+1-i,p}+\delta_{k+1-i,l}+\delta_{i,l}-\delta_{i,p}\bigr],
$$
so the ratio (\ref{Ci}) remains central in this case as well.

To address the problem of actual Poisson dimension, it is technically more convenient to remove the last monodromy datum $M_{2s,2s-1}$ from the basis and add $M_B$ to it. Then all remaining matrices are independent and we can take the restriction on $M_B$ into account explicitly. We again evaluate the Poisson bi-vector at the point in which all $M_{2j,2j-1}$ with $1\le j<s$ are diagonal and $M_B$ has nonzero anti-diagonal and the upper half-diagonal. We let $S[m_{i,j}]$ denote the orbit of the matrix element $m_{i,j}$ under the action of the symmetry group generated by $S_1[m_{i,j}]=m_{k+1-i,j}$ and $S_2[m_{i,j}]=m_{j,i}$. For a generic $(i,j)$ such an orbit comprises eight elements for $M_{2j,2j-1}$ and four elements for $M_B$ because of the reduction. It is a cumbersome calculation, omitted, to demonstrate that the Poisson algebra is non-degenerate for the sets $\cup_{j=1}^{s-1}S[m_{{2j,2j-1}_{k,l}}]\cup S[m_{B_{k,l}}]$ with $1\le l<k\le [(k+1)/2]$ and is highly degenerate for the sets
$\cup_{j=1}^{s-1}S[m_{{2j,2j-1}_{i,i}}]\cup S[m_{B_{i,i}}]$ with $i=1,\dots,[k/2]$: the Poisson dimension of every such set comprising $4s-1$ elements is $2(s-1)$. Finally, the set $\cup_{j=1}^{s-1}S[m_{{2j,2j-1}_{(k+1)/2,(k+1)/2}}]\cup S[[m_{B_{(k+1)/2,(k+1)/2}}]$ has zero Poisson
dimension. So, the total Poisson codimension for even $k$ is $[k/2](2s+1)=sk+k/2$ and for odd $k$ it is $[k/2](2s+1)+s=sk+(k-1)/2$, that is,
we have $sk+[k/2]$ independent Casimirs, as expected.
\end{proof}

\subsubsection{Braid-group action in $\Sigma_{0,s+1,1}$}
Let us denote $M_{2r,2r-1}$ by $M_{(r)}$ for brevity. The braid group for $\pi_1(\Sigma_{0,s+1,1})$ is generated by the standard operators $B_j$:
\beq
B_j:\ \left\{ M_{(j-1)}\to M_{(j-1)} M_{(j)}M_{(j-1)}^{-1};\ M_{(j)}\to M_{(j-1)};\ M_{(k)}\to M_{(k)},\ k\ne j,j-1\right\}.
\label{BG}
\eeq

\begin{lemma}
The braid-group action (\ref{BG}) preserves the quantum commutation relations for monodromy data. In the semiclassical limit, it also obviously preserves the set of central elements $(\tr M_{(r)}^k)$.
\end{lemma}

\proof We have to verify the preservation of quantum commutation relations. Checking this for $M_{(k)}$ with $k\ne j,j-1$ is simple: such $M_{(k)}$ have commutation relations of the same form (\ref{21-43}) with $M_{(j)}$, $M_{(j-1)}$ and with all elements of the multiplicative non-Abelian group generated by the matrices $M_{(j)}$ and $M_{(j-1)}$. The new matrices $M'_{(j-1)}=M_{(j-1)} M_{(j)}M_{(j-1)}^{-1}$ and $M'_{(j)}=M_{(j-1)}$ must satisfy the same commutation relations as $M_{(j-1)}$ and $M_{(j)}$. For commutation relations of $M'_{(j)}$ with itself it trivially holds, so we have to check two nontrivial relations. The first one is
\begin{align}
&\sheet1{M'}_{(j-1)}R_{12} \sheet2{M'}_{(j)}R_{12}^{-1}= \bigl(\sheet1M_{(j-1)}\sheet1M_{(j)}\sheet1M{}_{(j-1)}^{-1}\bigr) R_{12} \sheet2M_{(j-1)} R_{12}^{-1}
=\sheet1M_{(j-1)}\sheet1M_{(j)}\underline{\sheet1M{}_{(j-1)}^{-1} R_{12} \sheet2M_{(j-1)} }R_{12}^{-1}
\nonumber\\
=&\sheet1M_{(j-1)}\underline{\sheet1M_{(j)} R_{12}^{-\text{T}} \sheet2M_{(j-1)} }R_{12}^{\text{T}} \sheet1M{}_{(j-1)}^{-1} R_{12}  R_{12}^{-1}
=\underline{\sheet1M_{(j-1)}R_{12}^{-\text{T}} \sheet2M_{(j-1)} R_{12}^{\text{T}}  }\sheet1M_{(j)} R_{12}^{-\text{T}}  R_{12}^{\text{T}} \sheet1M{}_{(j-1)}^{-1}
\nonumber\\
=&R_{12} \sheet2M_{(j-1)} R_{12}^{-1} \sheet1M_{(j-1)}  \sheet1M_{(j)} \sheet1M{}_{(j-1)}^{-1}=R_{12} \sheet2{M'}_{(j)} R_{12}^{-1} \sheet1{M'}_{(j-1)}.
\nonumber
\end{align}
And the second relation is
\begin{align}
&R_{12}^{\text{T}}\sheet1{M'}_{(j-1)} R_{12}^{-\text{T}} \sheet2{M'}_{(j-1)}=
R_{12}^{\text{T}}\sheet1M_{(j-1)}\sheet1M_{(j)}\underline{\sheet1M{}_{(j-1)}^{-1} R_{12}^{-\text{T}} \sheet2M_{(j-1)}}\sheet2M_{(j)}\sheet2M{}_{(j-1)}^{-1}
\nonumber\\
=&R_{12}^{\text{T}}\sheet1M_{(j-1)}\sheet1M_{(j)}R_{12}^{-\text{T}} \sheet2M_{(j-1)} R_{12}^{-1} \underline{\sheet1 M{}_{(j-1)}^{-1} R_{12} \sheet2M_{(j)}}\sheet2M{}_{(j-1)}^{-1}
\nonumber\\
= &R_{12}^{\text{T}}\sheet1M_{(j-1)}\underline{\sheet1M_{(j)}R_{12}^{-\text{T}} \sheet2M_{(j-1)} }(R_{12}^{-1} R_{12}) \sheet2M_{(j)} R_{12}^{-1} \sheet1 M{}_{(j-1)}^{-1} R_{12} \sheet2M{}_{(j-1)}^{-1}
\nonumber\\
=&\bigl(\underline{R_{12}^{\text{T}}\sheet1M_{(j-1)} R_{12}^{-\text{T}} \sheet2M_{(j-1)}}\bigr) \bigl(\underline{R_{12}^{\text{T}} \sheet1M_{(j)}R_{12}^{-\text{T}}   \sheet2M_{(j)}}\bigr) R_{12}^{-1} \underline{\sheet1 M{}_{(j-1)}^{-1} R_{12} \sheet2M{}_{(j-1)}^{-1}}
\nonumber\\
=&\sheet2M_{(j-1)}\underline{R_{12}^{-1}\sheet1M_{(j-1)} R_{12} \sheet2M_{(j)} }R_{12}^{-1} \sheet1M_{(j)}(R_{12} R_{12}^{-1}) R_{12}^{-\text{T}} \sheet2M{}_{(j-1)}^{-1} R_{12}^{\text{T}} \sheet1 M{}_{(j-1)}^{-1} R_{12}
\nonumber\\
=&\sheet2M_{(j-1)}\sheet2M_{(j)} R_{12}^{-1}\sheet1M_{(j-1)} (R_{12}  R_{12}^{-1})\underline{ \sheet1M_{(j)} R_{12}^{-\text{T}} \sheet2M{}_{(j-1)}^{-1} }R_{12}^{\text{T}} \sheet1 M{}_{(j-1)}^{-1} R_{12}
\nonumber\\
=&\sheet2M_{(j-1)}\sheet2M_{(j)} R_{12}^{-1}\underline{\sheet1M_{(j-1)} R_{12}^{-\text{T}} \sheet2M{}_{(j-1)}^{-1} R_{12}^{\text{T}} }\sheet1M_{(j)} (R_{12}^{-\text{T}} R_{12}^{\text{T}} )\sheet1 M{}_{(j-1)}^{-1} R_{12}
\nonumber\\
=&\sheet2M_{(j-1)}\sheet2M_{(j)} (R_{12}^{-1}R_{12})\sheet2M{}_{(j-1)}^{-1}R_{12}^{-1} \sheet1M_{(j-1)} \sheet1M_{(j)} \sheet1 M{}_{(j-1)}^{-1} R_{12}
=\sheet2{M'}_{(j-1)} R_{12}^{-1} \sheet1{M'}_{(j-1)}R_{12}
\nonumber
\end{align}
\endproof

\subsubsection{IHX-relations}
We now probe the algebra of monodromy data corresponding to intersecting arcs. Our basic example is the case where we have a single intersection of arcs in  the case of $\Sigma_{0,s+1,1}$. We let $M_{(k)}$ denote as in the preceding subsection the monodromy datum  with endpoints $(2k,2k-1)$; all monodromy data start and terminate at the same bordered cusp. Let us consider the products of monodromy data $M_{(1)}M_{(2)}$ and $M_{(2)}M_{(3)}$ and take their product $\bigl(\sheet1{M_{(1)}}\sheet1{M_{(2)}}\bigr) R_{12}^{-\text{T}} \bigl(\sheet2{M_{(2)}}\sheet2{M_{(3)}}\bigr)$. For this product, we have:
\begin{align}
&\sheet1{M_{(1)}}\underline{\sheet1{M_{(2)}} R_{12}^{-\text{T}} \sheet2{M_{(2)}}}\sheet2{M_{(3)}}
=\underline{\sheet1{M_{(1)}} R_{12}\sheet2{M_{(2)}} R_{12}^{-1}}\sheet1{M_{(2)}} R_{12}^{-\text{T}} \sheet2{M_{(3)}}
\nonumber\\
=&R_{12}\sheet2{M_{(2)}} R_{12}^{-1}\sheet1{M_{(1)}}\sheet1{M_{(2)}} \underline{R_{12}^{-\text{T}} }\sheet2{M_{(3)}}
=R_{12}\sheet2{M_{(2)}} R_{12}^{-1}\sheet1{M_{(1)}}\sheet1{M_{(2)}} \bigl( q R_{12}-(q^{3/2}-q^{-1/2}) P_{12}\bigr)\sheet2{M_{(3)}}
\nonumber\\
=&q  R_{12}\sheet2{M_{(2)}} R_{12}^{-1}\underline{ \sheet1{M_{(1)}}\sheet1{M_{(2)}} R_{12}\sheet2{M_{(3)}} }
-(q^{3/2}-q^{-1/2}) R_{12}\sheet2{M_{(2)}} R_{12}^{-1}\sheet1{M_{(1)}}\sheet1{M_{(2)}} \underline{P_{12} \sheet2{M_{(3)}}}
\nonumber\\
=&q  R_{12}\bigl(\sheet2{M_{(2)}}\sheet2{M_{(3)}}\bigr) R_{12}^{-1} \bigl(\sheet1{M_{(1)}}\sheet1{M_{(2)}}\bigr) R_{12}
-(q^{3/2}-q^{-1/2}) R_{12}\bigl(\sheet2{M_{(2)}}\bigr) R_{12}^{-1}\bigl(\sheet1{M_{(1)}}\sheet1{M_{(2)}} \sheet1{M_{(3)}}\bigr) P_{12}
\label{M12-M23}
\end{align}
In this calculation, we use two identities: the first one is
$$
q^{1/2}R_{12}-q^{-1/2}R_{12}^{-\text{T}}=(q-q^{-1})P_{12},
$$
where $P_{12}=\sum_{i,j} \sheet1 e_{i,j}\otimes \sheet2 e_{j,i}$ is the standard classical permutation matrix. For this matrix, we have that
$$
\sheet1 M P_{12}= P_{12} \sheet2 M\quad\hbox{and}\quad  \sheet2 M P_{12}= P_{12} \sheet1 M
$$
for any (irrespectively classical or quantum) matrix $M$. Formula (\ref{M12-M23}) is similar to the three-term IHX-relation that is common in models of directed intersecting paths in knot theory.
{(This (local) relation is an unrooted version of the Jacobi identity in the theory of finite type (or Vassiliev) invariants of knots, links and 3-manifolds.)} In the right-hand side of (\ref{M12-M23}) we have two terms: one is the term in which the original monodromy data enter in opposite order, the other is the term containing new matrices: $M_{(2)}$ and $M_{(1)}M_{(2)}M_{(3)}$ corresponding to \emph{nonintersecting} paths; these new constitutive monodromy data enjoy commutation relations (\ref{32-41}). Disregarding the $R$-matrix structures, we schematically depict the IHX-relations in the following form (the over-/under-crossing indicates which monodromy datum  stands to the right):
$$
\begin{pspicture}(-6,-1)(6,1){
\newcommand{\CURVEONE}{%
{\psset{unit=1}
\psbezier[linewidth=1.5pt](0,0)(-1,-1)(-1,-2)(0,-2)
\psbezier[linewidth=1.5pt]{->}(0,0)(1,-1)(1,-2)(0,-2)
}
}
\newcommand{\CURVETWO}{%
{\psset{unit=1}
\psbezier[linewidth=6pt,linecolor=white](0,0)(-1,-1)(-1,-2)(0,-2)
\psbezier[linewidth=6pt,linecolor=white](0,0)(1,-1)(1,-2)(0,-2)
\psbezier[linewidth=1.5pt](0,0)(-1,-1)(-1,-2)(0,-2)
\psbezier[linewidth=1.5pt]{->}(0,0)(1,-1)(1,-2)(0,-2)
\pscircle[linewidth=1pt,fillstyle=solid, fillcolor=white,linecolor=white](0,0){0.3}
}
}
\newcommand{\CURVETHREE}{%
{\psset{unit=1}
\psbezier[linewidth=1.5pt](0,0)(-1.5,-0.62)(-1.5,-2)(0,-2)
\psbezier[linewidth=1.5pt]{->}(0,0)(1.5,-0.62)(1.5,-2)(0,-2)
}
}
\newcommand{\CURVEFOUR}{%
{\psset{unit=1}
\psbezier[linewidth=1.5pt](0,0)(-0.41,-1)(-0.5,-1.7)(0,-1.7)
\psbezier[linewidth=1.5pt]{->}(0,0)(0.41,-1)(.5,-1.7)(0,-1.7)
\pscircle[linewidth=1pt,fillstyle=solid, fillcolor=white,linecolor=white](0,0){0.3}
}
}
\put(-5.5,0){\makebox(0,0)[rc]{\hbox{{$q^{1/2}$}}}}
\put(-1.4,0){\makebox(0,0)[rc]{\hbox{{$-q^{-1/2}$}}}}
\put(3.5,0){\makebox(0,0)[rc]{\hbox{{$=(q-q^{-1})$}}}}
\rput{-22.5}(-4,1){\CURVEONE}
\rput{22.5}(-4,1){\CURVETWO}
\rput{22.5}(0,1){\CURVEONE}
\rput{-22.5}(0,1){\CURVETWO}
\rput(5,1){\CURVETHREE}
\rput(5,1){\CURVEFOUR}
}
\end{pspicture}
$$

\subsection{Monodromy algebras for $\Sigma_{0,2,2}$}

We consider the example of monodromy data for an ``eye'' -- the disc with the hole inside and with two bordered cusps on the outer boundary
(Fig.~\ref{fig:sigma}(b)).
{This case is of particular interest as it is a generalization of the modromy data associated to the Dubrovin connection in  the theory of Frobenius manifolds \cite{Dub}.}
In this case we have two monodromy data $M_1$ and $M_2$, both subject to restriction (\ref{starstar}). We then have the following statement
about Poisson leaves of this algebra.

\begin{lemma}
\label{lm:S022}
For $\Sigma_{0,2,2}$ with two cusps on one hole, the algebra of restricted monodromy data $M_1$ and $M_2$ has maximum Poisson dimension
$k(k-1)$. The $2k$ Casimirs are $\tr [[M_1M_2]^p]$, $p=1,\dots,k$, and the ratios $m_{1_{i,k+1-i}}/m_{2_{k+1-i,i}}$, $i=1,\dots,k$.
\end{lemma}

\begin{proof}
Anti-diagonal elements of both restricted monodromy matrices have homogeneous commutation relations with all other elements; it is
easy to check that ratios $m_{1_{i,k+1-i}}/m_{2_{k+1-i,i}}$ Poisson commute with all elements. Besides that we have that the product
$M_2M_1$ is a monodromy data corresponding to circumnavigating the central hole, so traces of all powers of this product are central elements.
(Because of the trace property, it is irrelevant whether we take the monodromy data to be $M_2M_1$ or $M_1M_2$.) We then again consider Poisson
bi-vectors over special matrices $M_1$ and $M_2$ with only halves of main diagonal and both anti-diagonals nonzero; we again split elements into
octuplets (orbits of $S[m_{1_{i,j}}]\cup S[m_{2_{i,j}}]$ with $1\le j<i\le [k/2]$), quadruplets (orbits of $S[m_{1_{\frac{k+1}{2},j}]}\cup S[m_{2_{\frac{k+1}{2},j}}]$ with $1\le j< [(k+1)/2]$), sextets (orbits of $S[m_{1_{i,i}}]\cup S[m_{2_{i,i}}]$ with $1\le i\le [k/2]$)
and doublets $\{m_{1_{\frac{k+1}{2},\frac{k+1}{2}}},m_{2_{\frac{k+1}{2},\frac{k+1}{2}}}\}$. Octuplets and quadruplets are non-degenerate, sextets
have Poisson dimension two and doublets have Poisson dimension zero, so the total Poisson co-dimension is $4[k/2]$ for even $k$ and $4[k/2]+2$
for odd $k$; it is easy to see that it is $2k$ in both cases, as expected. 
\end{proof}

\subsection{Monodromy algebra for $\Sigma_{0,1,3}$}

The last example pertains to another elementary building block of monodromy data: the ideal triangle $\Sigma_{0,1,3}$ (Fig.~\ref{fig:sigma}(c)) that for $k=2$ corresponds to the  modromy data of the Airy equation. In this case, we have two monodromy matrices $M^1$ and $M^2$ (here $M^2$ follows $M^1$), both having the reduced (upper-anti-triangular) form - we denote them with an upper index to distinguish them from the other examples. Plus we
have to take into account that their product, $M^2M^1$, has itself lower-anti-triangular form. This imposes $k(k-1)/2$ restrictions on
entries of $M^1$ and $M^2$; for the general position situation it is not difficult to see that we can express all non-anti-diagonal
entries of, say, matrix $M^2$ in terms of entries of $M^1$ and the anti-diagonal elements of $M^2$ and, using relations (\ref{basic1}), we obtain
that, whereas entries of the matrix $M^1$ enjoy commutation relations (\ref{3.100}), the Poisson relations for entries of $M^1$ and 
anti-diagonal entries of $M^2$
\beq
\label{M1M2}
\{m^1_{i,j},m^2_{r,k+1-r}\}=m^1_{i,j}m^2_{r,k+1-r}\delta_{i,k+1-r},\ i+j\le k+1,\ r=1,\dots,k,
\eeq
are homogeneous and all elements of $M^1$ belonging to the same row commute in the same way with all $m^2_{k,k+1-r}$. All $m^2_{k,k+1-r}$
mutually commute. We then have the following statement about Casimirs of the algebra of the set of elements $\{m^1_{i,j},\ i+j\le k+1\}\cup
\{m^2_{k,k+1-r},\ r=1,\dots,k\}$. 

\begin{lemma}
\label{lm:S013}
For $\Sigma_{0,1,3}$ with three bordered cusps on a disc (one outer hole), the algebra of restricted monodromy data $M^1$ and 
$m^2_{r,k+1-r}$ has maximum Poisson dimension $\frac {k(k+1)}{2}+k-\Bigl[\frac{k}{2}\Bigr]$. The $\Bigl[\frac{k}{2}\Bigr]$ Casimirs are 
$$
C_d=\frac{[M^1]_d^{\text{UL}}[M^1]_{k-d}^{\text{UL}}}{\prod_{i=1}^d [m^1_{i,k+1-i}]^2 \prod_{i=d+1}^{k-d} [m^1_{i,k+1-i}]}\cdot
\frac{\prod_{i=1}^d [m^2_{k+1-i,i}]^2} 
{\prod_{i=1}^d [m^2_{i,k+1-i}]^{2}},\quad d=1,\dots, \Bigl[\frac{k}{2}\Bigr],
$$
where $[M^i]_d^{\text{UL}}$, $i=1,2$ denote the upper-left minors of size $d$ of the matrix $M^i$.
\end{lemma}

\proof The fact that $C_d$ are Casimirs can be verified directly (see Fig.~\ref{fig:M013}). We then again consider
a Poisson bi-vector over the pattern in which nonzero elements are $m^1_{i,i}$ with $i\le [(k+1)/2]$, $m^1_{i,k+1-i}$, and $m^2_{i,k+1-i}$.
For quadruples of the matrix $M^1$ the Poisson brackets are non-degenerate, whereas the Poisson dimension of quintuplets $\{m^1_{i,i},m^1_{i,k+1-i},m^1_{k+1-i,i},m^2_{i,k+1-i},m^2_{k+1-i,i}\}$ for $i=1,\dots,[k/2]$ is four (so each quintuplet adds one Casimir) and the Poisson dimension of the doublet $\{m^1_{\frac{k+1}{2},\frac{k+1}{2}},m^2_{\frac{k+1}{2},\frac{k+1}{2}}\}$ is two; the total Poisson codimension therefore matches the above number of Casimirs. \endproof

\begin{figure}[tb]
{\psset{unit=0.7}
\begin{pspicture}(-3,-3)(3,3)
\psframe[linecolor=black, linewidth=1pt, fillstyle=solid, fillcolor=white](-2.5,2.5)(2.5,-2.5)
\psframe[linecolor=blue, linestyle=dashed, linewidth=2pt, fillstyle=solid, fillcolor=white](-2.5,2.5)(1,-1)
\psframe[linecolor=blue, linestyle=dashed, linewidth=2pt, fillstyle=solid, fillcolor=yellow](-2.4,2.4)(-1,1)
\pcline[linewidth=1pt](-2.5,-2.5)(2.5,2.5)
\pcline[linewidth=1pt](-1,1)(-1,-1)
\pcline[linewidth=1pt](-1,1)(1,1)
\pcline[linewidth=2pt, linecolor=blue](2.35,2.25)(1.05,.95)
\pcline[linewidth=2pt, linecolor=blue](2.45,2.15)(1.15,.85)
\pcline[linewidth=2pt, linecolor=blue](-0.75,-.95)(.95,0.75)
\put(-1.75,1.75){\makebox(0,0)[cc]{\hbox{{$-2$}}}}
\put(-1.75,0){\makebox(0,0)[cc]{\hbox{{$0$}}}}
\put(-2,-1.5){\makebox(0,0)[cc]{\hbox{{$+2$}}}}
\put(-0.25,0.25){\makebox(0,0)[cc]{\hbox{{$0$}}}}
\put(1.5,2){\makebox(0,0)[cc]{\hbox{{$-2$}}}}
\put(0,1.75){\makebox(0,0)[cc]{\hbox{{$-2$}}}}
\put(0,-2.9){\makebox(0,0)[cc]{\hbox{{\bf (a)}}}}
%
\end{pspicture}
\begin{pspicture}(-3,-3)(3,3)
\psframe[linecolor=black, linewidth=1pt, fillstyle=solid, fillcolor=white](-2.5,2.5)(2.5,-2.5)
\pcline[linewidth=1pt](-2.5,-2.5)(2.5,2.5)
\pcline[linewidth=1pt](-1,2.5)(-1,-1)
\pcline[linewidth=1pt](-2.5,1)(1,1)
\pcline[linewidth=1pt](1,2.5)(1,1)
\pcline[linewidth=1pt](-2.5,-1)(-1,-1)
\pcline[linewidth=3pt, linecolor=red](-2.2,-2.4)(-.9,-1.1)
\pcline[linewidth=3pt, linestyle=dashed, linecolor=red](2.4,2.2)(1.1,.9)
\put(-1.75,1.75){\makebox(0,0)[cc]{\hbox{{$+2$}}}}
\put(-1.75,0){\makebox(0,0)[cc]{\hbox{{$0$}}}}
\put(-2,-1.5){\makebox(0,0)[cc]{\hbox{{$-2$}}}}
\put(-0.25,0.25){\makebox(0,0)[cc]{\hbox{{$0$}}}}
\put(1.5,2){\makebox(0,0)[cc]{\hbox{{$+2$}}}}
\put(0,1.75){\makebox(0,0)[cc]{\hbox{{$+2$}}}}
\put(0,-2.9){\makebox(0,0)[cc]{\hbox{{\bf (b)}}}}
%
\end{pspicture}
}
\caption{Constructing Casimirs for the restricted matrix $M^1$  subject to Poisson algebra (\ref{3.100}) and the anti-diagonal entries of
the matrix $M^2$. Numbers in
the corresponding rectangles or triangles indicate the sign of homogeneous commutation relations between elements of $M^1$ in
the corresponding region and (a) the products of minors $[M^1]_d^{\text{UL}}[M^1]_{k-d}^{\text{UL}}$ ($d\le k-d$) divided
by the special products of anti-diagonal elements $\prod_{i=1}^d [m^1_{i,k+1-i}]^2 \prod_{i=d+1}^{k-d} [m^1_{i,k+1-i}]$;
 (b) by the ratios of products $\prod_{i=1}^d\bigl[ [m^2_{k+1-i,i}]^2 [m^2_{i,k+1-i}]^{-2} \bigr]$
 of elements on the antidiagonal of the matrix $M^2$.
We see that these signs are complementary in all regions. All anti-diagonal entries of $M^2$ mutually commute and commute with the above ratios of elements of $M^1$ because every term contains equal number of elements from the same row of the matrix $M^1$ in the numerator and denominator.}
\label{fig:M013}
\end{figure}
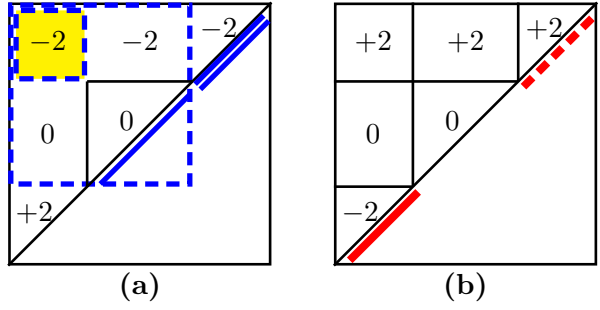

\section{The extended Riemann--Hilbert correspondence}\label{se:analytic}
\label{s:RH}

In this section we conjecture how to define a suitably decorated moduli space $\mathcal Z_{irr}^k( \Sigma_{g,s})$ of irregular connections $\nabla$ on a holomorphic rank $k$-vector
bundle $E \to \Sigma_{g,s}$ in such a way that the Riemann--Hilbert correspondence
$$
RH:\mathcal M_{irr}^k( \Sigma_{g,s})\to \mathcal M_{g,s,n}^k
$$
is a Poisson isomorphism.

Given an irregular  connection  $\nabla$ on a holomorphic rank $k$-vector
bundle $E \to \Sigma_{g,s}$,
choosing a coordinate $z$ for $\Sigma_{g,s}$ amounts to giving a linear system of differential equations $\nabla_{\frac{\partial}{\partial z}}$ with $s$ poles $a_1,\dots,a_s$  of Poincar\'e rank $r_1,\dots,r_s$ or in other words a meromorphic matrix-valued differential
$$
d- A(z) dz
$$
with fixed multiplicities $r_1+1,\dots, r_s+1$ at $a_1,\dots,a_s$.

For the sake of simplicity, we restrict to the non-ramified case where $A(z)$ is diagonalisable at each $a_p$. Then, Krichever proved that  the space $\mathcal E_{irr}(r_1,\dots,r_s)$ of all possible meromorphic matrix-valued differentials of this form modulo the $SL_k$ action is $(2g+s+r_1+\dots+r_s-2)(k^2-1)$-dimensional and is foliated in symplectic leaves by fixing the exponents, with the symplectic form:
$$
\omega= -\frac{1}{2}\sum_{t=1}^{kg} {\rm res}_{\gamma_t}{\rm Tr}\left(Y^{-1} \delta A\wedge \delta Y\right) -\frac{1}{2}\sum_{p=1}^{s} {\rm res}_{a_p}{\rm Tr}\left(Y^{-1}_p \delta A\wedge \delta Y_p\right)
$$
where $Y_p$ is the formal local solution of  $dY=A(z) Y dz $ at $a_p$ and $\gamma_1,\dots,\gamma_{gk}$ are the simple zeroes of the holomorphic sections of the vector bundle
$E \to \Sigma_{g,s}$ \cite{Kri}. This symplectic form induces a Poisson structure on $\mathcal Z_{irr}^k( \Sigma_{g,s}):= \mathcal E_{irr}(r_1,\dots,r_s)\times \mathbb C^{(r_1+\dots+r_s)(k-1)}$, where $\mathbb C^{(r_1+\dots+r_s)(k-1)}$ is the space of decorations, or in other words a choice of growth rates of the absolute value of the formal solutions $Y_p$ modulo polynomial growth \cite{STWZ}. Following the ideas by Gaiotto, Moore and Neitzke \cite{GMN}, we impose $n=2(r_1+\dots+r_s)$, so  that
$\dim(\mathcal Z_{irr}^k( \Sigma_{g,s}))=\dim( \mathcal M_{g,s,n}^k)$. We end this paper with the following

{\noindent{\bf Conjecture:}
{\it The Riemann--Hilbert correspondence
$$
RH:\mathcal Z_{irr}^k( \Sigma_{g,s})\to \mathcal M_{g,s,n}^k
$$
is a Poisson isomorphism.}}
 
We have tested this conjecture in the case of $g=0$, $k=2$ and connections with only one irregular singularity of Poincar\'e rank $3$. In this case we have the Jimbo-Miwa linear system associated to the PII equation and $n=6$. The decorated character variety $\mathcal M_{0,1,6}^2$ has dimension $9$ with one Casimir. As explained in \cite{CMR}, the isomonodromicity condition means that we need to restrict to a $2$--dimensional sub--algebra in  $\mathcal M_{0,1,6}^2$ defined by the set of functions that Poisson commute with the frozen cluster variables corresponding to arcs connecting pairs of bordered cusps. On the l.h.s. of the Riemann-Hilbert correspondence, the space $\mathcal Z_{irr}^2( \Sigma_{0,1})$ has also dimension $9$ and by imposing the isomonodromicity condition (where $t$ is the PII independent variable):
$$
\frac{\partial A}{\partial t}-\frac{\partial B}{\partial z} =[B,A]
$$
one obtains a restriction to a $2$--dimensional space  \cite{HR} which we denote  $\tilde{\mathcal Z}_{irr}^2( \Sigma_{0,1})$.

\begin{remark}
The space  $\tilde{\mathcal Z}_{irr}^2( \Sigma_{0,1})$ is the  de Rham side of the Riemann-Hilbert correspondence. Recently \cite{Sz} a complete description of the two-dimensional (family of) holomorphic symplectic moduli spaces of rank 2 Higgs bundles over $\mathbb P^1$ having a unique pole of order 4 as singularity, and regular leading-order term was obtained. This moduli  (for a fixed choice of parameters) can be interpreted as the Dolbeault counterpart  ${\mathcal M}_{D,irr}^2$ of  $\tilde{\mathcal Z}_{irr}^2( \Sigma_{0,1}).$ \end{remark}

\begin{remark}
A general extended Riemann-Hilbert correspondence for related objects (a moduli space of stable unramified  irregular singular parabolic connections on smooth projective curves and a set
$\widetilde{\mathcal R(g,k,s)}$ of generalized monodromy data coming from topological monodromies, formal monodromies and Stokes data) was proposed by M. Inaba and M. Saito in \cite{IS}. They proved that the moduli space of generalized monodromy data is a nonsingular affine scheme $\mathcal R(g,k,s)$ given by the categorical quotient $\mathcal R(g,k,s) = \widetilde{\mathcal R(g,k,s)}//G$ for a natural action of a reductive group $G.$ 
An immediate comparison of dimensions for $\mathcal M_{g,s,n}^k$ and  $\widetilde{\mathcal R(g,k,s}$ shows a good correspondence. For example, in the case of $\mathcal M_{0,1,6}^2$ (whose dimension is 9) the dimension of  $\widetilde{\mathcal R(0,2,1)}$ is 8 (9 minus one Casimir) and $\dim \mathcal M_{0,1,6}^2 = \dim \mathcal R(0,2,1)=2.$
\end{remark}

\section{Acknowledgements}
The authors wish to dedicate this paper to Nigel Hitchin, whose beautiful mathematics has inspired many of our papers. We thank A.Alekseev, A. Glutsuk, M. Gualtieri, A. Shapiro, M. Shapiro, N. Nekrasov, B. Pym and P. Severa for helpful discussions.
The research of VR was partially supported by  the Russian Foundation for Basic Research under grant RFBR-15-01-05990. He is thankful to  MPIM (Bonn) and SISSA (Trieste), where a part of this work was done, for invitation and excellent working conditions. The work of L.Ch. was partially financially supported by the Russian Foundation for Basic Research under grant RFBR-15-01-99504.


\end{document}